\documentclass[a4paper,12pt]{amsart}
\usepackage[margin=3.cm]{geometry}
\usepackage{verbatim}
\usepackage{yhmath}
\usepackage{hyperref} 
\usepackage{lipsum}
\usepackage{amsthm}
\usepackage{amsmath}
\usepackage{amsbsy}
\usepackage{bigints}
\usepackage{amsfonts}%
\usepackage{amssymb}%
\usepackage{graphicx}
\usepackage[utf8]{inputenc}%
\usepackage[T1]{fontenc}%
\usepackage[english]{babel}%
\usepackage[all]{xy}
\usepackage{amsmath, amsthm, amssymb,amsfonts,mathrsfs,amscd,latexsym}
\usepackage{graphicx}
\usepackage{tikz-cd}
\usepackage{mathtools}
\usepackage{hyperref}
\usepackage{graphicx}
\graphicspath{ {./image/} }
\usepackage{CJKutf8}
\allowdisplaybreaks[4]
\usepackage[dvipsnames]{xcolor}

\newcommand{\poubelle}[1]{}

\usepackage{wrapfig}
\usepackage[all]{xy}
\usepackage{colortbl}
\usepackage{amsmath,amssymb,graphics,epsfig,color}
\usepackage{enumerate,wasysym,stmaryrd}
\usepackage{color}
\usepackage{fancybox}
\usepackage{colordvi}
\usepackage{multicol}
\usepackage{colordvi}
\usepackage{geometry}
\usepackage{lscape}
\usepackage{amsthm}
\usepackage{eqnarray,amsmath}
\usepackage{array}
\usepackage{booktabs}   
  \newcolumntype{x}[1]{>{\centering\hspace{0pt}}p{#1}}
\def\A{\mathcal{A}}
\def\X{\mathfrak{X}}

\def\R{\mathbb{R}}

\newtheorem{theorem}{Theorem}[section]

\newtheorem{theoremA}{Theorem}

\newtheorem{corollary}[theorem]{Corollary}

\newtheorem{corollaryA}{Corollary}

\newtheorem{define}[theorem]{Definition}
\newtheorem{example}[theorem]{Example}
\newtheorem{lemma}[theorem]{Lemma}
\newtheorem{prop}[theorem]{Proposition}
\newtheorem{remark}[theorem]{Remark}

\newcommand{\DPL}{\mathcal{DPL}(\Delta)}
\newcommand{\RPL}{\mathcal{RPL}(\Delta)}

\newcommand{\aut}{\mathrm{Aut}}
\newcommand{\red}{\mathrm{red}}

\newcommand{\T}{\mathbb{T}}

\newcommand{\C}{\mathbb{C}}
\renewcommand{\P}{\mathbb{P}}
\newcommand{\ext}{\mathrm{ext}}
\newcommand{\DF}{\mathcal{F}^{\Delta}_{v,w}}

\renewcommand{\A}{\mathcal{A}}
\newcommand{\Y}{\mathcal{Y}}
\renewcommand{\X}{\mathcal{X}}
\newcommand{\J}{\textnormal{\textbf{J}}^{\textnormal{NA}}}
\renewcommand{\S}{\mathbb{S}}

\newcommand{\hess}{\textnormal{Hess}}

\newcommand{\Z}{\mathcal{Z}}
\newcommand{\Scal}{\mathrm{Scal}}

\newcommand{\tl}{\mathbb{T}}
\newcommand{\Gc}{G^{\mathbb{C}}}

\usepackage[utf8]{inputenc}

\allowdisplaybreaks

\title[Effective Yau-Tian-Donaldson conjecture for projective bundles]{On the effective Yau--Tian--Donaldson conjecture for projective bundles}

\AtEndDocument{%
{\footnotesize

\textsc{Simon Jubert, Dipartimento di Matematica ``Guido Castelnuovo'',
Sapienza Universit\`a di Roma, Piazzale Aldo Moro 5,
00185 Roma, Italy}\par
\textit{Email address}: \texttt{simon.jubert@uniroma1.it}

\bigskip

\textsc{Chenxi Yin, Département de Mathématiques, UQAM,
C.P. 8888, Succursale Centre-ville, Montréal
(Québec), H3C 3P8, Canada}\par
\textit{Email address}: \texttt{yin.chenxi@courrier.uqam.ca, yin.chenxi.mathematics@gmail.com}

}
}

\begin{document}
\raggedbottom

\author{Simon Jubert}
\author{Chenxi Yin}

\begin{abstract}
Consider the projectivization $\mathbb{P}(E)$ of a holomorphic vector bundle $E$ over a smooth curve, equipped with a fixed K\"ahler class $[\omega]$. Our first main result establishes that the existence of an extremal K\"ahler metric in $[\omega]$ is equivalent to the weighted stability of a moment polytope $\Delta$ with respect to explicit weight functions depending on the topological data of $\mathbb{P}(E)$.

We then introduce a class of test configurations, which we call compatible test configurations. These are constructed by gluing horospherical test configurations of the fibers, arising from convex functions on $\Delta$ via Delcroix's construction, to the principal bundle associated with $\mathbb{P}(E)$. Our second main result shows that the weighted stability of $\Delta$ naturally arises when computing the relative uniform K-stability of $[\omega]$ with respect to compatible test configurations.
\end{abstract}
\maketitle

\tableofcontents

\section{Introduction}
\subsection{Background}
A fundamental problem in Kähler geometry, proposed by Calabi~\cite{EC82}, is the search for a \textit{canonical Kähler metric} representing a given Kähler class on a compact Kähler manifold. Calabi proposed as natural candidates Kähler metrics whose scalar curvature defines a Hamiltonian vector field preserving the Kähler structure. These are known as \textit{extremal Kähler metrics} (extremal metrics for short).

\smallskip

Recent developments have established that the existence of an extremal Kähler metric in a fixed Kähler class is equivalent to various notions of stability \cite{BJ25, DZ25, Li22, trusiani2026}. These results are  known as the Yau--Tian--Donaldson (YTD) correspondence.

Despite these advances, finding effective criteria for testing stability remains a major open problem in the field. Indeed, determining whether a given compact Kähler manifold is uniformly K-stable is generally a difficult task, and several approaches have been developed over the past decades. One such approach involves the introduction of numerical invariants, such as the delta invariant (see, e.g., \cite{BJ20, CRZ19, DJ25, Z21, Zha24}).

Another approach exploits the symmetries of the manifold to translate K-stability into a stability condition on a suitable moment polytope (see, e.g., \cite{ACGT11, Del2023, Del25, DJ23, Don10, Jub23, LLS23, Sze06, ZZ}). The present paper focuses on this latter approach for projective bundles over a curve.

\bigskip
The existence of extremal K\"ahler metrics on projective bundles over curves has been studied extensively, see e.g.  \cite{ACGT06, ACGT11, AK, fuj92, Hong1999, HwangSinger2002, hong2002, RollinSinger2005, RT06, Ton98}. Establishing an effective criterion for the existence of extremal K\"ahler metrics in a given K\"ahler class for these manifolds remains an important open problem. Some special cases have been addressed in the literature, and we briefly recall some of them below.

\smallskip

A key result of \cite[Theorem 2]{AK} states that the relative uniform K-stability of $(Y,[\omega_Y])$ forces the underlying vector bundle $E$ to decompose as a direct sum of projectively flat bundles $E_j$, so that
\begin{equation}{\label{splitting}}
Y=\mathbb{P}\left(\oplus_{j=0}^{\ell}E_j\right)\longrightarrow C.   
\end{equation}

We may therefore restrict our attention to bundles of this form.

When $\ell=0$, the problem of finding an extremal metric is completely settled by \cite{ACGT11, RT06}, which show that a cscK metric exists in every K\"ahler class (see also \cite[Theorem 1 and Remark 3]{AK}). 

When $\ell=1$, building on the Calabi ansatz, Apostolov, Calderbank, Gauduchon and T{\o}nnesen-Friedman developed the theory of admissible and compatible K\"ahler metrics, establishing several existence and stability results for projective bundles \cite{ACGT06, ACGT08, ACGT11}. A key feature of their construction is that the special geometric structure of the decomposition $\ell=1$  allows them to reduce the extremal K\"ahler equation from a PDE to an ODE on a one-dimensional moment polytope.

The approach developed in the present paper applies to all $\ell\geq 1$ and provides an explicit criterion for the existence of extremal K\"ahler metrics in any fixed K\"ahler class, in terms of weighted stability of a suitable moment polytope.

\subsection{Main results and strategy of proof} We fix a Kähler class $[\omega_Y]$ on $Y$. Without loss of generality, we assume that $[\omega_Y]$ restricts to the Kähler class of the Fubini-Study metric on each fiber.

This splitting \eqref{splitting} induces an action of an $\ell$-dimensional compact torus $\T$ on $Y$, where each $S^1$ factor of $\T$ acts diagonally on the corresponding factor $E_i$. Assuming that $\omega_Y$ is $\T$-invariant, the action of $\T$ on $(Y,\omega_Y)$ is Hamiltonian, with moment map

\begin{equation*}
    \mu_{\omega_Y} : Y \longrightarrow \Delta,
\end{equation*}
where the moment polytope $\Delta$ is the $\ell$-simplex.

Our main result is the following

\begin{theoremA}{\label{t:A}}
There exists an extremal K\"ahler metric in $(Y,[\omega_Y])$ if and only if $\Delta$ is $(p\bar{p},\bar{w})$-uniformly K-stable. Moreover, any extremal K\"ahler metric in $[\omega_Y]$ is compatible.
\end{theoremA}

In the above statement, $p, \bar{p}, \bar{w}$ are explicit weight functions on $\Delta$ determined by the topological data of $Y$.

The implication from existence to stability was established in \cite{Jub26}. We now outline the main ideas behind the converse implication.

Compatible K\"ahler metrics, introduced in \cite{ACGT06, ACGT04}, generalize the metrics on $\mathbb{P}^1$-bundles arising from the classical Calabi ansatz. In our setting, a symplectic potential $u$ on $\Delta$ determines a compatible K\"ahler metric $\omega_X$ in the Fubini--Study class $[\omega_{\mathrm{FS}}]$ on the fiber
\[
X := \mathbb{P}\left(\oplus_{j=0}^\ell \mathbb{C}^{d_j}\right)
\]
of $Y$.  By the construction of \cite{ACGT06}, if $\omega_X$ is $(\bar{p},\hat{w})$-weighted cscK, then the associated compatible metric on $Y$ is extremal. Thus, the existence problem reduces to finding a $(\bar{p},\hat{w})$-weighted cscK metric in $[\omega_{\mathrm{FS}}]$ induced by a symplectic potential $u$ on $\Delta$, where $\hat{w}$ is another explicit weight determined by the topological data of $Y$.

To establish the existence of such a metric, we first show that $(p\bar{p},\bar{w})$-uniform K-stability of $\Delta$ implies the coercivity of the weighted Mabuchi energy $\mathcal{M}^{\Delta}_{p\bar{p},\bar{w}}$ on $\Delta$.

The main difficulty is that the weight $p$ may vanish on the boundary of $\Delta$, preventing us from directly applying the usual existence results for weighted cscK metrics of \cite{DJL25, JubertPan26}. To overcome this difficulty, we transfer the coercivity of $\mathcal{M}^{\Delta}_{p\bar{p},\bar{w}}$ to the $(\bar{p},\hat{w})$-weighted Mabuchi energy $\mathbf{M}_{\bar{p},\hat{w}}^X$ on $(X,[\omega_{\mathrm{FS}}])$, restricted to the space of compatible K\"ahler metrics induced by symplectic potentials $u$ on $\Delta$. The key point is that the weight $\bar{p}>0$ on $\Delta$.

Finally, we adapt the arguments of \cite[Theorem~6.1]{Jub23}, together with the results of \cite{DJL25, DJL24} (see also \cite{HanLiu}), to establish the existence of a $(\bar{p},\hat{w})$-weighted cscK metric in $[\omega_{\mathrm{FS}}]$ induced by a symplectic potential $u$ on $\Delta$. Importantly, this conclusion holds even though $\mathbf{M}_{\bar{p},\hat{w}}^X$ need not be coercive on the full space of $\T$-invariant K\"ahler metrics in $[\omega_{\mathrm{FS}}]$.

\bigskip

Our second main result provides an interpretation of the weighted K-stability of $\Delta$ in terms of relative uniform K-stability with respect to a suitable class of test configurations. More precisely, we construct a class of test configurations $(\Y_f,A_f)$ for $(Y,[\omega_Y])$, associated with certain convex functions $f$ on $\Delta$, and show that relative uniform K-stability with respect to these test configurations is equivalent to the weighted uniform K-stability of $\Delta$. The construction proceeds in two main steps:

\begin{enumerate}
    \item We construct suitable test configurations $\mathcal{X}_f$ for the fiber $X$, called \textit{compatible} test configurations, which admit both toric and horospherical symmetries. These test configurations arise from certain convex functions $f$ on $\Delta$.
    \item We glue these test configurations fiberwise, using their horospherical symmetries, to obtain the desired \textit{compatible} test configurations $\mathcal{Y}_f$ for $Y$.
\end{enumerate}

Suppose that the class $[\omega_Y]$ is polarized by a holomorphic line bundle $L_Y$. We then establish the following result.

\begin{theoremA}\label{t:B}
The K\"ahler class $(Y,[\omega_Y])$ is $\T$-relatively uniformly K-stable with respect to compatible test configurations if and only if $\Delta$ is $(p\bar{p},\bar{w})$-uniformly K-stable.
\end{theoremA}

We outline the main ideas of the proof below.

Building on an extension of the generalized Calabi construction to non-commutative structure groups, we construct a K\"ahler metric $\omega_{\mathcal{Y}_f}$ on $\mathcal{Y}_f$ from a suitable K\"ahler metric $\omega_{\mathcal{X}_f}$ on $\mathcal{X}_f$. The latter is obtained by embedding $\mathcal{X}_f$ into a sufficiently large projective space and restricting the Fubini--Study metric.

We then show that the fiber-wise $\T \times \mathbb{S}^1$-action on $(\mathcal{Y}_f,\omega_{\mathcal{Y}_f})$ is semisimple and rigid in the sense of \cite{ACGT04, ACGT11}. This structure allows us to explicitly compute the Donaldson--Futaki invariant and the reduced non-Archimedean $J$-functional of $(\mathcal{Y}_f,\omega_{\mathcal{Y}_f})$.

\medskip

In particular, combining Theorem \ref{t:A} with Theorem \ref{t:B} shows that the existence of an extremal K\"ahler metric is entirely characterized by stability with respect to compatible test configurations. More precisely, we obtain the following corollary.

\begin{corollaryA}\label{c:intro:B}
The following statements are equivalent.

\begin{enumerate}
    \item There exists an extremal metric in $(Y,2\pi c_1(L_Y))$.
    \item $(Y,2\pi c_1(L_Y))$ is $\T$-relatively K-stable on compatible test configuration
    \item $\Delta$ is $(p\bar{p},\bar{w})$-uniformly $K$-stable.
\end{enumerate}
\end{corollaryA}

\subsection{Organization of the paper}
This paper is organized as follows.
\begin{itemize}
  \item Section 2 reviews the basic definitions and results concerning weighted cscK metrics and weighted K-stability.

  \item In Section 3, we construct compatible test configurations $\mathcal{X}_f$ for $X$ and relate their weighted K-stability to the weighted uniform K-stability of $\Delta$.

  \item Section 4 establishes the existence of compatible weighted cscK metrics on $X$ under a suitable weighted stability condition on $\Delta$. As an application, we complete the proof of Theorem \ref{t:A}.

  \item In Section 5, we study the horospherical symmetries of $\mathcal{X}_f$ and establish the rigid and semisimple structure required for the construction of test configurations on projective bundles.

  \item In Section 6, we construct compatible test configurations for $Y$ and compute their Donaldson--Futaki invariants and reduced non-Archimedean $J$-functionals. We then establish Theorem \ref{t:B} and deduce the corresponding characterization of extremal K\"ahler metrics in terms of compatible test configurations.

  \item Finally, Appendix \ref{a:toric} reviews weighted toric K-stability. Appendix \ref{a:toric:comp} extends the existence result of Section 4 to a broader class of toric manifolds with rigid and semisimple actions. Appendix \ref{a:horo} recalls basic facts about horospherical varieties, while Appendix \ref{bundleconstruction} reviews the construction of K\"ahler metrics on associated bundles.
\end{itemize}

\section*{Acknowledgment}
 The authors wish to express their  gratitude to Vestislav Apostolov, Sébastien Boucksom, Thibaut Delcroix, Mattias Jonsson, Julien Keller, Eveline Legendre, and Pietro Mesquita Piccione for enlightening discussions. They are particularly grateful to Vestislav Apostolov, Partha Ghosh, Mattias Jonsson and Julien Keller for their valuable comments on an earlier version of this manuscript, and to Paul Gauduchon for generously sharing his unpublished notes on multiplicity-free manifolds.

During the preparation of this work the first author is funded by the ERC SiGMA - 101125012 (PI: Eleonora
Di Nezza).

The authors sincerely thank the referee for their careful reading of our manuscript and valuable comments, which have helped improve its presentation.

\section{Review on weighted csck metrics and weighted K-stability}{\label{s:wcsck}}

\subsection{Weighted cscK metrics}{\label{s:reviewcscK}}

Let $(X, \omega)$ be a compact K\"ahler manifold  of complex dimension $n$ and let $\mathbb{T}$ be an $\ell$-dimensional compact real torus in the reduced automorphism group $\mathrm{Aut}_{\mathrm{red}}(X)$ (see \cite[Section 2.4]{Gau}).  We denote by $\mathfrak{t}$ its Lie algebra and by $\mathfrak{t}^*$ its dual vector space. Hence $\mathbb{T}=\mathfrak{t}/\Lambda$, where $\Lambda$ is the lattice of the torus. Suppose $\omega$ is a $\mathbb{T}$-invariant K\"ahler form. It is well-known that the $\mathbb{T}$-action on $X$ is $\omega$-Hamiltonian (see e.g. \cite[Section 2.5]{Gau}) and  that $\Delta_{\omega}:= \mu^\T_{\omega}(X)$ is a convex polytope in $\mathfrak{t}^{*}$ \cite{Ati, GS, Lah19}. Here $\mu^\T_\omega: X \rightarrow \mathfrak{t}^{*}$ is a moment map associated with $\omega$. When no confusion between groups can arise, we will simply write $\mu_\omega$ instead of $\mu^\T_\omega$. 
We can always normalize $\mu_{\omega}$ in such a way that $\Delta=\Delta_{\omega}$ is independent of the $\T$-invariant K\"ahler metric in the K\"ahler class $[\omega]$, see \cite[Lemma 1]{Lah19}.

For a given positive  function $v \in {C}^{\infty}(\Delta,\R_{>0})$, we define \textit{the $v$-weighted Ricci form} by

\begin{equation*}
    \mathrm{Ric}_v(\omega) := \mathrm{Ric}(\omega) -\frac{1}{2}dd^c \log v(\mu_\omega).
\end{equation*}
Then the $v$-weighted scalar curvature of a $\T$-invariant K\"ahler metric $\omega$ is defined by

\begin{equation*}
 {\Scal_v(\omega)}:= 2\Lambda_{\omega,v}(\mathrm{Ric}_v(\omega)),
\end{equation*}

\noindent where for any closed $\T$-invariant $(1,1)$-form $\beta$ with $\T$-moment map $\mu_\beta$, $\Lambda_{\omega,v}(\beta):=\Lambda_{\omega}(\beta)+ \langle d\log(v),\mu_\beta \rangle$  is the $v$-weighted trace with respect to $\omega$. When $[\beta]=2\pi c_1(X)$, we use the canonical normalization of its moment map induced by the natural lift of the $\T$-action to the anticanonical bundle $K_X^{-1}$ of $X$. The definition of the weighted scalar curvature is equivalent to the one introduced in \cite{Lah19}, up to a multiplicative factor $v$. 

Given a second smooth function $w \in {C}^{\infty}(\Delta,\R)$, a $\T$-invariant K\"ahler metric $\omega$ is called \emph{$(v,w)$-weighted cscK} if its $v$-weighted scalar curvature satisfies

\begin{equation}{\label{wcsck}}
    \Scal_v(\omega)=w(\mu_\omega).
\end{equation}

In some cases, including the one of central importance in this paper, it is relevant to allow $v$ to vanish on the boundary of the polytope $\Delta$, see Example~\ref{e:proj:bund}. In this situation, expressions involving $\log v$ and its derivatives, including the preceding pointwise formula for the weighted scalar curvature, are understood on the inverse image of the interior of $\Delta$.

The significance of \eqref{wcsck} in relation to various geometric conditions is thoroughly described in \cite{Lah19}. However, we shall mention a few specific cases below which are relevant to this paper. 

\begin{example}[Extremal K\"ahler metric]
    For $v = 1$, and $w=\ell_{\ext}$ is the affine extremal function, the $(v,w)$-weighted cscK metrics are Calabi's extremal K\"ahler metrics \cite{EC82}.
\end{example}

\begin{example}{\label{e:semisimple:principal}}\textnormal{(Semisimple principal fibrations)}
Let $(X,\omega,\T)$ be a compact K\"ahler manifold and $P$ a $\T$-principal bundle over a product of cscK manifolds $B=\prod_a B_a$, endowed with a semisimple connection. We can then consider the associated semisimple principal fibration, given by
\[
Y:=P\times_\T X\longrightarrow B.
\]
Then, by \cite{ACGT04,AJL,Jub23}, any extremal metric on $Y$ corresponds to a weighted $(p,\tilde{w})$-cscK metric on $X$, where $p$ and $\tilde{w}$ are certain weights associated with the fibration. Indeed an important point here is that any extremal metric on such a bundle is compatible, thanks to \cite[Corollary 6.6]{Jub23}.
\end{example}

We present below a key example for the purpose of the present paper.

\begin{example}{\label{e:proj:bund}}\textnormal{(Projective bundle over a curve \cite{ACGT11})}
We refer to \cite[Section 3.2]{ACGT11} or \cite{AK} for more details.

Let $Y:=\mathbb{P}(E)$ be the projectivization of a holomorphic vector bundle $E$ over a curve $C$. By a result of Apostolov--Keller \cite[Theorem 2]{AK}, $E$ decomposes as a direct sum of projectively flat vector bundles $E_j$, i.e.
\[
Y = \mathbb{P}\left( \oplus_{j=0}^\ell E_j \right) \longrightarrow C,
\]
provided that there exists a relatively uniformly K-stable K\"ahler class on $Y$. We let 
    \(X := \mathbb{P}\left( \oplus_{j=0}^\ell \mathbb{C}^{d_j} \right)
\) be the fiber of $Y$, where
$d_j = \mathrm{rk}(E_j)$, $G=\prod_{j=0}^\ell U(d_j)$.

Consider the projective space
\begin{equation*}
  \mathbb{P}^\ell:=\mathbb{P}(\mathbb{C}^{\ell+1})
\end{equation*}
endowed with its standard $\T$-action, where $\T \cong (\mathbb{S}^1)^\ell$, and let $\Delta$ be its standard moment simplex.

The manifolds $X$ and $\mathbb{P}^\ell$ are related through a certain blow-up $\tilde{X}$ of $X$.  More precisely, we first blow up $X$ along the projective subspace defined by the vanishing of the first component in the direct sum decomposition defining $X$. We then blow up the strict transform of the projective subspace defined by the vanishing of the second component, and proceed similarly for each remaining component. The resulting manifold
\begin{equation*}
\tilde{X}:=\mathbb{P}\left(\oplus_{j=0}^\ell \mathcal{O}_{\mathbb{P}(\mathbb{C}^{d_j})}(-1)\right) \longrightarrow \prod_{j=0}^\ell \mathbb{P}(\mathbb{C}^{d_j})
\end{equation*}
is a bundle with fiber $\mathbb{P}^\ell$.

We now explain how the extremal K\"ahler equation on $Y$ reduces to a weighted cscK equation on $X$, and subsequently to a weighted cscK equation on $\mathbb{P}^\ell$.

By the generalized Calabi construction of \cite{ACGT11}, a compatible K\"ahler metric on $Y$ is extremal if and only if the associated compatible metric $\omega_X$ on $X$ is $(\bar{p},\hat{w})$-weighted cscK, where
\begin{equation}{\label{bar:p:w}}
\begin{aligned}
\bar{p}(x)
&= c-\sum_{j=0}^\ell \mu(E_j)L_j(x), \\
\hat{w}(x)
&= \ell_{\mathrm{ext}}(x)
 - \frac{4(1-\mathrm{g})}{\bar{p}(x)} .
\end{aligned}
\end{equation}

Here $c$ is a constant such that $\bar{p}>0$ on $\Delta$, $\mathrm{g}$ is the genus of the curve $C$, and $\mu(E_j):=\frac{\mathrm{deg}(E_j)}{d_j}= \int_C c_1(E_j)/d_j$ is the slope of the bundle $E_j$. Moreover, $\ell_{\mathrm{ext}}$ is the affine extremal function of $(Y,[\omega_Y],\hat{T})$, where $\hat{T}$ is a maximal compact torus inside $\mathrm{Aut}_{\mathrm{red}}(Y)$.

We now use the compatible construction of $X$  to reduce this weighted equation to an equation on $\mathbb{P}^\ell$. By Lemma \ref{l:scal}, the compatible metric $\omega_X$ associated with a toric K\"ahler metric $\omega_{\mathbb{P}^\ell}$ is $(\bar{p},\hat{w})$-weighted cscK if and only if $\omega_{\mathbb{P}^\ell}$ is $(p\bar{p},\bar{w})$-weighted cscK, where
\begin{equation}{\label{weight:x}}
\begin{aligned}
p(x)
&= \prod_{j=0}^\ell (L_j(x))^{d_j-1}, \\
\bar{w}(x)
&= \hat{w}(x)
 - \sum_{j=0\atop d_j \neq 1}^{\ell}
 \frac{2d_j(d_j-1)}{L_j(x)} .
\end{aligned}
\end{equation}

Here $\{L_j\}_j$ denotes the label of $\Delta$, see Section \ref{s:torique}, and $2d_j(d_j-1)$ is the constant scalar curvature of the Fubini--Study metric $\omega_{\mathrm{FS}_j}$ on $\mathbb{P}(\mathbb{C}^{d_j})$. 

Thus, the existence of an extremal compatible K\"ahler metric on $Y$ is reduced to the existence of a $(p\bar{p},\bar{w})$-weighted cscK metric on $\mathbb{P}^\ell$. Equivalently, the existence of an extremal compatible K\"ahler metric on $Y$ is equivalent to the existence of a $(\bar{p},\hat{w})$-weighted cscK metric on $X$.
\end{example}

 There are crucial geometric differences between the situations in Examples~\ref{e:semisimple:principal} and~\ref{e:proj:bund}. Notably, in Example~\ref{e:proj:bund} the associated principal bundle has a non-commutative structure group
$G := \prod_j U(d_j),$ and the corresponding weight function may vanish along the boundary of the polytope.

\subsection{Weighted K-stability}{\label{s:red:j}}

We first recall the usual definition of an equivariant K\"ahler test configuration. 

\begin{define}
A $\mathbb{T}$-equivariant Kähler test configuration $(\mathcal{X}, \mathcal{A})$  for a K\"ahler manifold $(X, [\omega], \mathbb{T})$ is a normal compact Kähler space $\mathcal{X}$ endowed with
\begin{itemize}
  \item a flat morphism $\pi : \mathcal{X} \to \mathbb{P}^1$;
  \item a $\mathbb{C}^*$-action $\rho$ covering the standard $\mathbb{C}^*$-action on $\mathbb{P}^1$, and a $\mathbb{T}$-action commuting with $\rho$ and preserving $\pi$;
  \item a $\mathbb{T} \times \mathbb{C}^*$-equivariant biholomorphism $\Pi_0 : (\mathcal{X} \setminus \pi^{-1}(0)) \cong (X \times (\mathbb{P}^1 \setminus \{0\}))$;
  \item a Kähler class $\mathcal{A} \in H^{1,1}(\mathcal{X}, \mathbb{R})$ such that $(\Pi_0^{-1})^*(\mathcal{A})|_{X \times \{\tau\}} = [\omega]$.
\end{itemize}
\end{define} 

We say that $(\mathcal{X}, \mathcal{A})$ is \emph{smooth} if $\mathcal{X}$ is smooth and \emph{dominating} if $\Pi_0$ extends to a $\mathbb{T} \times \mathbb{C}^*$-equivariant morphism
\begin{equation*}
\Pi : \mathcal{X} \to X \times \mathbb{P}^1.
\end{equation*}

For any K\"ahler class $\alpha$ we let $\alpha^{[k]}:=\alpha^k/k!$. For any dominating test configuration $(\X,\A)$ we define the non-Archimedean $J$-functional using the convention  by:

\begin{equation*}
 \J(\X,\A):=\frac{1}{\int_X[\omega]^{[n]}} \left(\int_{\X}(\Pi^* [\omega])^{[n]} \wedge \mathcal{A}-\int_{\X}\mathcal{A}^{[n+1]}\right).
\end{equation*}

We introduce the $\xi$-twist $(\mathcal{X}_\xi, \mathcal{A}_\xi)$ of a smooth test configuration $(\mathcal{X}, \mathcal{A})$, see \cite{Has16, Li_2022}). Let $\xi \in 2\pi\Lambda$ be a lattice element. The discussion below also holds for a rational element $\xi$ of $\mathfrak{t}$. The choice of $\xi$ induces a $\mathbb{C}^*$-action $\rho_{\xi}$ on $\mathcal{X}$. By composing the original $\mathbb{C}^*$-action $\rho$ with $\rho_{\xi}$, we obtain a new $\mathbb{C}^*$-action on the space $\X \setminus \pi^{-1}(\infty)$. Then $(\mathcal{X}_\xi, \mathcal{A}_\xi)$ is defined by trivially compactifying $\X \setminus \pi^{-1}(\infty)$ at infinity with respect to the new $\mathbb{C}^*$-action $\rho \cdot \rho_\xi$  (see \cite[Definition 2.7]{BHJ2017} for a precise definition of the compactification). We let $\S^1_\xi$ be the $\S^1$-action on $\X_\xi$ obtained by composing the original $\S^1$-action with the one induced by $\xi$. Then $(\mathcal{X}_\xi, \mathcal{A}_\xi)$ is a $\T \times \S^1_\xi$-equivariant test configuration.

When \(\xi\) is not rational, the pair \((\mathcal{X}_{\xi}, \mathcal{A}_{\xi})\) does not necessarily define a test configuration. However, the functional \(\mathbf{J}^{\text{NA}}(\mathcal{X}_{\xi}, \mathcal{A}_{\xi})\) can be obtained as a continuous extension of the corresponding quantity defined for rational values of \(\xi\).

\begin{define}[\cite{Has16, Li_2022}]
For any dominating test configuration $(\mathcal{X},  \mathcal{A})$, the reduced non-Archimedean $\mathbf{J}$-functional is defined by

\begin{equation}{\label{non:arc:j}}
\mathbf{J}^{\textnormal{NA}}_{\mathbb{T}^{\mathbb{C}}} (\mathcal{X}, \mathcal{A}) := \inf_{\xi \in \mathfrak{t}} \mathbf{J}^{\textnormal{NA}} (\mathcal{X}_{\xi}, \mathcal{A}_{\xi}).
\end{equation}
\end{define}

\begin{define}[\cite{Lah19}]
For any smooth test configuration $(\mathcal{X},  \mathcal{A})$ the weighted Futaki invariant is defined by

\begin{equation*}
    \mathbf{DF}_{v,w}(\mathcal{X},\A)= - \int_{\mathcal{X}}\big(\Scal_v(\Omega_\X) - w(\mu_{\Omega_\X})\big)v(\mu_{\Omega_\X})\Omega_\X^{[n+1]}+(8 \pi) \int_X v(\mu_\omega)\omega^{[n]},
\end{equation*}
where $\Omega_\X$ is any $\T$-invariant K\"ahler metric in $\A$ with $\Delta$-normalized moment map $\mu_{\Omega_\X}$.
\end{define}

Let $\mathcal{C}$ be a class of test configurations for $(X,[\omega],\T)$, i.e. a subset of $\T$-equivariant test configurations for $(X,[\omega],\T)$. 

\begin{define}
    We say that a compact K\"ahler manifold $(X,[\omega],\T)$ is 
    
    \begin{enumerate}
        \item $(v,w)$-weighted K-stable on $\mathcal{C}$ if there exists a constant $\lambda>0$ such that for any dominating smooth test configuration $(\mathcal{X},\A) \in \mathcal{C}$

    \begin{equation*}
        \mathbf{DF}_{v,w}(\mathcal{X},\A) \geq \lambda \mathbf{J}^{\textnormal{NA}}_{\mathbb{T}^{\mathbb{C}}} (\mathcal{X}, \mathcal{A})
    \end{equation*}
    \item relatively uniformly K-stable, if $v=1$ and $w=\ell_{\mathrm{ext}}$ is the affine extremal function in $(1)$ above.
    \end{enumerate}

\end{define}

In this paper, we shall consider the following special classes of smooth, dominating test configurations.

\begin{example}

\begin{enumerate}
    \item  Compatible test configuration $(\X_f,\mathcal{A}_{\X_f})$ for the projective space $X=\mathbb{P}(\oplus_{j=0}^\ell \mathbb{C}^{d_j})$ see Sections \ref{s:compatible:test} and \ref{horo tc for proj}.
    \item Compatible test configuration $(\Y_f,\mathcal{A}_{\Y_f})$ for projective bundles over a curve $\mathbb{P}(\oplus_{j=0}^\ell E_j) \rightarrow C$, see Section \ref{s:bundle:tc}.
\end{enumerate}  
\end{example}

\subsection{Weighted toric K-stability}\label{s:torique}

Let $(Z^\ell,\omega,\T)$ be a compact toric K\"ahler manifold and let $(\Delta,\mathbf{L})$ be its associated labeled Delzant polytope.  More precisely, $\mathbf{L}$ is a collection of affine-linear functions defining $\Delta$. Moreover, a polytope is said to be Delzant if $\Delta$ is compact; $\Delta$ is \emph{simple}, meaning that each vertex of $\Delta$ annihilates 
    precisely $\ell$ of the affine functions in $\mathbf{L}$ and the corresponding inward normals form a basis of $\mathfrak{t}$; $\Delta$ is \emph{integral}, meaning that for each vertex  of $\Delta$, the inward normal vectors of the adjacent facets span $\Lambda$.

Following \cite{SKD, Lah19, LLS}, for $v \in \mathcal{C}^{\infty}( \Delta,\R_{\geq 0})$, $v>0$ on $\mathring{\Delta} $, and $w$ smooth on the interior of $\Delta$, with $vw$ extending smoothly to $\Delta$, we define the $(v,w)$-Donaldson--Futaki invariant of the labeled polytope $(\Delta,\textbf{L})$ as the functional 
\(\DF : \mathcal{C}^{0}(\Delta) \rightarrow \mathbb{R}\) by
\begin{equation}{\label{define-futaki}}
  \DF(f):=  2\int_{\partial \Delta}f v d\sigma -  \int_{\Delta}  f  w  v dx.
\end{equation}
where $d\sigma$ is the induced measure on each face $F_i \subset \partial \Delta$ by letting $dL_i \wedge d \sigma = -dx$ and $dx$ is the Lebesgue measure on $\Delta$. We suppose that the weights and their product satisfy the regularity assumptions above and that $
    \mathcal{F}_{v,w}(f)=0 $ 
for all $f$ affine-linear on $\Delta$. Integration by parts (see e.g. \cite{SKD}) reveals that the latter condition is  necessary  for the existence of a $(v,w)$-cscK metric in $(Z,[\omega_Z],\mathbb{T})$. 
Following \cite{Has16}, we introduce the $v$-weighted $J$-norm:
\begin{equation*}
\| f\|_{J_v(\Delta)}:= \inf_{\xi\in \mathrm{Aff}(\Delta)}  \int_\Delta \left(f+\xi-\inf_\Delta(f+\xi)\right) v dx,
\end{equation*}
 where \(\mathrm{Aff}(\Delta)\) denotes the space of affine functions on \(\Delta\). We recall the usual definition of weighted K-stability for a polytope, first introduced in the unweighted setting in \cite{SKD}.

 Let $\mathcal{CV}^{\infty}(\Delta) $ be the space of continuous convex functions on $\Delta$, smooth in the interior $\mathring{\Delta}$.  

\begin{define}{\label{d:stability}}
The labeled polytope \((\Delta,\textbf{L})\) is \((v,w)\)-uniformly K-stable if 
there exists $\lambda > 0$ such that for all  $f \in \mathcal{CV}^{\infty}(\Delta)$, 
\begin{equation*}
\mathcal{F}_{v,w}^{\Delta}(f) \geq \lambda \| f \|_{J_v(\Delta)}.
\end{equation*}
\end{define}

We introduce the $v$-weighted $L^1$-norm

\begin{equation*}
\| f \|_{L_v^1(\Delta)}:= \int_{\Delta} |f|vdx.   
\end{equation*}

When $v=1$, we simply write $\| \cdot \|_{L^1(\Delta)}$ for the unweighted $L^1$-norm and $\| \cdot \|_{J(\Delta)} $   for the unweighted $J$-norm.

Following \cite{SKD}, we fix $x_0 \in \mathring{\Delta}$  and consider the following  normalization
 
 \begin{equation}{\label{normalized-function-polytope}}
     \mathcal{CV}^{\infty}_*(\Delta):=\{ f \in \mathcal{CV}^{\infty}(\Delta) \text{ } | \text{ } f(x) \geq f(x_0)=0 \}.
 \end{equation}
 
\noindent Then, any $f\in \mathcal{CV}^{\infty}(\Delta)$ can
be written uniquely as $f = f^*+f_0$, where $f_0$ is affine-linear and $f^*\in \mathcal{CV}_*^{\infty}(\Delta)$ is the linear projection from $ \mathcal{CV}^{\infty}(\Delta)$ to $\mathcal{CV}^{\infty}_*(\Delta) $.

\begin{lemma}{\label{l:stab:eq}}
Let $v \in \mathcal{C}^{\infty}( \Delta,\R_{\geq 0})$ be strictly positive on the interior of $\Delta$. There exists $C>0$ such that for any $f \in \mathcal{CV}^{\infty}(\Delta)$. We have the following equivalence of norms:

\begin{equation*}
        \| f \|_{J_{v}(\Delta)} \leq \| f^* \|_{L_{v}^1(\Delta)} \leq C  \| f \|_{J_{v}(\Delta)}.  
\end{equation*}
\end{lemma}

When $v=1$, the above Lemma is well-known, see e.g. \cite[Proposition 5.4.1.]{NS}. A straightforward adaptation shows that the statement still holds in the context of Lemma \ref{l:stab:eq}.

\section{Compatible test configurations of projective spaces}{\label{s:comp:tc}}

\subsection{Blown-down structure and compatible K\"ahler metrics}{\label{s:set-up}}

We consider 
\(
X = \mathbb{P}\!\left(\oplus_{j=0}^{\ell} \mathbb{C}^{d_j}\right),
\)
equipped with its standard toric $T$-action.
We let $\hat{\Delta}$ be the associated moment polytope. We have a block diagonal action of a torus $\T\cong (\mathbb{S}^1)^{\ell}$ on $X$ as a sub-torus of the standard torus $T$ : the $j$th circle of $\T$ is induced by the scalar multiplication on the factor $\mathbb{C}^{d_j}$.  
Let
\begin{equation*}
    \tilde{X}:=\mathbb{P}\left(\oplus_{j=0}^\ell \mathcal{O}_{\mathbb{P}(\mathbb{C}^{d_j})}(-1)\right) \longrightarrow \prod_{j=0}^\ell \mathbb{P}(\mathbb{C}^{d_j})
\end{equation*}
be the manifold obtained from $X$ by the sequence of blow-ups described in Example~\ref{e:proj:bund}.
 On each projective space $\mathbb{P}(\mathbb{C}^{d_j})$ with $d_j\geq2$, we consider the Fubini-Study metric $\omega_{\mathrm{FS}_j}$. Let $
    P \longrightarrow \prod_{j=0}^\ell \mathbb{P}(\mathbb{C}^{d_j})$ be the $\T$-bundle associated with the complex bundle  $ \oplus_{j=1}^\ell \bigg(\mathcal{O}_{\mathbb{P}(\mathbb{C}^{d_j})}(-1)\otimes\mathcal{O}_{\mathbb{P}(\mathbb{C}^{d_0})}(1)\bigg)\rightarrow \prod_{j=0}^\ell \mathbb{P}(\mathbb{C}^{d_j})$. We fix a basis $(\xi_1,\dots,\xi_\ell)$ of the Lie algebra $\mathfrak{t}$ of $\T$. Then, there is a connection $\theta$ on $P$ with curvature 
\begin{equation*}
    d\theta = \sum_{j=0 \atop
    d_j \neq 1}^{\ell} p_j \otimes \omega_{\mathrm{FS}_j}, \quad p_j \in \mathfrak{t},
\end{equation*}
where $p_j=\xi_j$ for $j=1,\dots,\ell$ and $p_0=(-1,\dots,-1)$. Hence,  $\tilde{X}$ is the semisimple principal fibration associated with $P$ and the $\ell$-projective space $(\mathbb{P}^{\ell},\omega_{\mathrm{FS}})$:
\begin{equation*}
    \tilde{X} := P \times_\T  \mathbb{P}^\ell \longrightarrow \prod_{j=0}^\ell \mathbb{P}(\mathbb{C}^{d_j}).
\end{equation*}
Let $\mu$ be the moment map of a toric K\"ahler metric $\omega_{\mathbb{P}^\ell} \in [\omega_{\mathrm{FS}}]$. By $\T$-invariance, $\mu$ descends to a smooth $\mathfrak{t}^*$-valued function on $\tilde{X}$. We then consider the semi-positive closed form 
\begin{equation}{\label{comp:metric}}
   \hat{\omega}=\sum_{j=0\atop
    d_j \neq 1}^\ell( \langle p_j, \mu \rangle + c_j) \omega_{\mathrm{FS}_j} + \langle d\mu \wedge \theta \rangle, \quad \text{ on } \tilde{X}
\end{equation}
where $\langle p_j, \mu \rangle + c_j =:L_j(\mu) \in \mathbf{L}$, the label of the standard simplex $\Delta$ of $(\mathbb{P}^\ell, \T, [\omega_{\mathrm{FS}}])$.  Here $\mathfrak{t}^*$ denotes the dual Lie algebra of $\mathbb{T}$.   Actually, it follows from \cite[Proposition 2]{ACGT06} that $\hat{\omega}$ is the pull-back of  
a unique K\"ahler metric $\omega$ on $X$
via the blow-down map $\mathrm{Bd} : \tilde{X} \rightarrow X$. We shall mention that the Fubini--Study metric is compatible.

Following \cite{ ACGT06, ACGT04} we define the following

\begin{define}{\label{d:comp:metric}}
    A K\"ahler metric $\omega$ on $X$ defined as above from a K\"ahler metric $\omega_{\mathbb{P}^\ell}$ on $\mathbb{P}^\ell$ is called a compatible K\"ahler metric.
\end{define}

Observe that by construction, the $\T$-action on $(X,\omega,\T)$ is Hamiltonian, with moment polytope $\Delta$.

\subsection{The Lie algebra and moment coordinates}
We fix a compatible K\"ahler metric $\omega$ on $X$ as defined in Definition \ref{d:comp:metric}. Let $x:=(x_1,\dots,x_{\ell})$ be the moment coordinates on $\Delta$ of the $\T$-action on $(X,\omega)$.  Let $T_B:=\prod_{j=0}^\ell T_j$, where $T_j \cong (\mathbb{S}^1)^{d_j-1}$ is the standard action on the projective space $ B_j:=\mathbb{P}(\mathbb{C}^{d_j})$. We let $B:=\prod_{j=0}^\ell B_j$, $\Delta_B$ be the moment polytope of $(B,T_B,\omega_B)$ and $\omega_B:= \sum_{j=0,
    d_j \neq 1}^\ell \omega_{\mathrm{FS}_j}$.

The set-up described above and  in Section \ref{s:set-up} is a special case of a toric manifold with rigid semisimple action in the terminology of \cite{ACGT06, ACGT04}, see Appendix \ref{a:toric:comp}. Hence, we deduce the following result from \cite[Section 3.4]{ACGT11}.

\begin{lemma}{\label{l:moment}}
The $T_B$-action on $(B,\omega_B)$ lifts to an isometric Hamiltonian action $\hat{T}_B \subset \aut_{\red}(X)$ with moment coordinates

\begin{equation}{\label{moment}}
    \hat{x}^j= x^j (\langle  p_j, x \rangle + c_j), \quad \text{for } j=0,\dots \ell, \text { and } d_j \neq 1,
\end{equation}
where $x^j$ are the Hamiltonian coordinates of $(B_j,\omega_{\mathrm{FS}_j},T_j)$. Moreover there is an exact sequence of Lie algebras

\begin{equation*}
  \{0\} \longrightarrow \mathfrak{t} \longrightarrow \hat{\mathfrak{t}} \longrightarrow \mathfrak{t}_B \longrightarrow \{0\},
\end{equation*}
where $\mathfrak{t}$ and $\mathfrak{t}_B$ are the Lie algebras of $\T$ and $T_B$ respectively. Moreover, the injection $ \mathfrak{t} \rightarrow \hat{\mathfrak{t}} $ yields the projection between moment polytopes
\begin{equation}{\label{proj}}
    \pi : \hat{\Delta} \longrightarrow \Delta.
\end{equation}
\end{lemma}
From the above Lemma, the moment coordinate $\hat{x}$ of $(X,\omega,T)$ can be written as
 
 \begin{equation*}
     \hat{x}=(x, \hat{x}_B),
 \end{equation*}

\noindent where $\hat{x}_B:=(\hat{x}_1^j,\dots, \hat{x}^j_{d_j-1})_{j=0, d_j\neq1}^{\ell+1}$ are the moment coordinates of the lifted action  $\hat{T}_{B}$ of $T_B$. This yields the  identification

\begin{equation}{\label{identify:pol}}
   \mathring{\hat{\Delta}} \cong \mathring{\Delta} \times \mathring{\Delta}_B
\end{equation}
through

\begin{equation*}
    \hat{x}=(x,\hat{x}_B) \rightarrow \left(x, \frac{1}{\langle p_j, x \rangle +c_j} \hat{x}^j\right)=(x,x_B), \quad j=0,\dots \ell,
\end{equation*}
where $\mathring{\Delta}$ denotes the interior of $\Delta$ and similarly for $   \mathring{\hat{\Delta}}$ and $\mathring{\Delta}_B$.

\subsection{The spaces of integral functions}
Let $v>0$ be a weight function on $\Delta$, considered also as a weight on $\hat{\Delta}$ via $\pi$. We consider the space $L_v^1(\hat{\Delta})$ of functions on $\hat{\Delta}$ integrable with respect to the Lebesgue measure $vd\hat{x}$.  From \eqref{moment} and \eqref{identify:pol}, we have that on $\mathring{\hat{\Delta}}$

\begin{equation}{\label{measure}}
   d\hat{x}=pdxdx_B,
\end{equation}
where $dx_B$ is the Lebesgue measure on $\Delta_B$  and 

\begin{equation}{\label{hatv}}
     p(x)  = \prod_{j=0}^\ell \big(\langle p_j,x \rangle +c_j  \big)^{d_j-1}.
 \end{equation}

Let $L_{pv}^1(\Delta)$ be the space of integrable functions on $\Delta$ with respect to $pvdx$. We deduce the following key Lemma:

\begin{lemma}
For any function $f \in L_{vp}^1(\Delta)$, 

\begin{equation}{\label{int:func}}
   \int_{\hat{\Delta}} f vd\hat{x} =   \mathrm{Vol}(\Delta_B) \int_{\Delta} f p vdx.
\end{equation}   
In particular, we have an embedding $\pi^* : L_{pv}^1(\Delta) \hookrightarrow L_v^1(\hat{\Delta})$.
\end{lemma}

In the above Lemma we identify the function $f$ and its pullback by $\pi$ to $\hat{\Delta}$. We will use this abuse of notation throughout the paper.

\begin{proof}
  By \eqref{identify:pol} and \eqref{measure}

\begin{equation*}
\begin{split}
\int_{\hat{\Delta}} fv d\hat{x} 
=& \int_{\mathring{\Delta} \times \mathring{\Delta}_B} f pv dx dx_B =\mathrm{Vol}(\Delta_B) \int_{\Delta} f pv dx,
\end{split}
\end{equation*}   
where $\mathrm{Vol}(\Delta_B)$ denotes the volume of $\Delta_B$ with respect to the Lebesgue measure.

\end{proof}

\subsection{Compatible symplectic potentials}
The following lemma was established in \cite[Lemma~8]{ACGL} in the absence of blow-downs. The same argument extends to our setting. For each $j$ with $d_j \ge 2$, we fix a symplectic potential $u_j$ for the Fubini--Study metric $\omega_{\mathrm{FS}_j}$.

\begin{lemma}{\label{Lemma-symp-pot}}
Let $u \in \mathcal{S}(\Delta, \mathbf{L})$ be the symplectic potential of a K\"ahler metric $\omega_{\mathbb{P}^\ell}$ on $\mathbb{P}^\ell$. Let $\omega$ be the compatible K\"ahler metric associated with $\omega_{\mathbb{P}^\ell}$. Then 
  \begin{equation}{\label{comp:symp}}
     \hat{u}= u  + \sum_{j=0\atop
    d_j \neq 1}^\ell (\langle p_j, x \rangle +c_j) u_j.
 \end{equation}
is a symplectic potential of $\omega$.
 \end{lemma}

We let $\mathcal{S}(\hat{\Delta}, \hat{\mathbf{L}})$ be the space of symplectic potentials of $\hat{\Delta}$.

\begin{define}{\label{d:comp:metric:pot}}
  A symplectic potential $\hat{u} \in \mathcal{S}(\hat{\Delta}, \mathbf{L})$ induced by a symplectic potential $u \in \mathcal{S}(\Delta, \mathbf{L})$ as in the above Lemma is called a compatible symplectic potential. 
\end{define}

By definition, a K\"ahler metric constructed from a compatible symplectic potential is a compatible K\"ahler metric in the sense of Definition \ref{d:comp:metric}.

 \subsection{The labeled polytope}
 For $d_j\geq2$, we let $\Delta_j$ be the standard $(d_j-1)$-simplex of $(\mathbb{P}(\mathbb{C}^{d_j}), \omega_{\mathrm{FS}_j},T_j)$. Let $\mathbf{L}_j=\{L^j_{0}, \dots, L^j_{d_j-1}\}$ be the label of $\Delta_j$. If $d_j=1$, we define $L^j_0=1$.  
\begin{lemma}{\label{l:label}}
The label $\hat{\mathbf{L}}=(\hat{L}_0,\dots, \hat{L}_{n})$ of the Delzant polytope $\hat{\Delta}$  is
\begin{equation}{\label{lab:delta:hat}}
 \hat{\Delta}= \{ ( \langle p_j , x \rangle + c_j) L^j_{i_j} \geq 0, \quad  j=0,\dots,\ell, \text{ and } i_j=0,\dots, d_j-1  \}.
\end{equation}
\end{lemma}

Observe that $\langle p_j , x \rangle + c_j) L^j_{i_j}$ is indeed an affine function of the moment coordinates $\hat{x}$, see \eqref{moment}.

\begin{proof}

We compute the RHS of \eqref{lab:delta:hat} and show that it corresponds to the label of the standard $n$-simplex $(\hat{\Delta}, \hat{\mathbf{L}})$. For each facet of $\Delta$ not involved in the blow-down, i.e., associated with an affine function $\langle p_j , x \rangle + c_j$ with $d_j=1$, we have $(\langle p_j , x \rangle + c_j) L^j_{i_j} = \langle p_j , x \rangle + c_j$.  Hence, we now suppose that $d_j \ge 2$. 

After reordering the labeling $\mathbf{L}$, if necessary, we assume that for every $k = 1, \dots, \ell$,
\(
\langle p_k, x \rangle + c_k = x_k\) and $\langle p_{0}, x \rangle + c_{0} = 1-\sum_{k=1}^{\ell}x_k.
\) Similarly, for $i=1, \dots, d_j - 1$, we have $L_i^j = x_i^j$ and $L_{0}^j = 1-\sum_{i=1}^{d_j-1} x_i^j$. Hence, for any $j = 0, \dots, \ell$ and $i = 1, \dots, d_j - 1$, by~\eqref{moment},
\[
(\langle p_j, x \rangle + c_j) L_i^j = \hat{x}_i^j.
\]

Also, for $j = 1, \dots, \ell$,
\[
(\langle p_j, x \rangle + c_j) L_0^j
= x_j \Bigl(1 - \sum_{i=1}^{d_j-1} x_i^j \Bigr)
= x_j - \sum_{i=1}^{d_j-1} \hat{x}_i^j.
\]

Finally, 
\begin{equation*}
\begin{aligned}
(\langle p_0, x \rangle + c_0) L_0^0
&= \left(1 - \sum_{k=1}^{\ell} x_k\right)
\Bigl(1 - \sum_{i=1}^{d_0-1} x_i^0 \Bigr) = 1 - \sum_{k=1}^{\ell} x_k  -  \sum_{i=1}^{d_0-1} \hat{x}_i^0
\end{aligned}
\end{equation*}

We now consider the following linear transformation:

\begin{equation}\label{transfo:pol}
\left\{
\begin{aligned}
\hat{x}_i^j &\longrightarrow \hat{y}_i^j,
&& j = 0,\dots,\ell,\quad i = 1,\dots,d_j-1, \\[0.4em]
x_k &\longrightarrow y_k + \sum_{i=1}^{d_k-1} \hat{y}_i^k,
&& k = 1,\dots,\ell.
\end{aligned}
\right.
\end{equation}

Hence \eqref{transfo:pol} identifies the RHS of \eqref{lab:delta:hat} and the standard $n$-simplex
\[
\Bigl\{
(\hat{y}_i^j, y_k)
\;\Big|\;
\hat{y}_i^j \ge 0,\;
y_k \ge 0,\;
1 - \sum_{k=1}^\ell y_k - \sum_{j=0}^\ell \sum_{i=1}^{d_j-1} \hat{y}_i^j \ge 0
\Bigr\},
\]
as the standard Delzant $n$-simplex. Above the indices are as follows: $j=0,\dots,\ell$, $i=1,\dots,d_j-1$, $k=1,\dots, \ell$ and by convention $\hat{y}_i^j=0$ if $d_j=1$.
\end{proof}

\subsection{The boundary measure}

We define the boundary measure $d \sigma_{\hat{\Delta}}$ on the boundary $\partial \hat{\Delta}$ of $\hat{\Delta}$ by
\begin{equation}{\label{bound:mes}}
    - d \hat{L}_c \wedge d\sigma_{\hat{\Delta}} |_{\{\hat{L}_c =0\}}   = d\hat{x},
\end{equation}
and we define the boundary measure $d\sigma_{\Delta_B}$  on $\partial \Delta_B$ similarly.

\begin{lemma}{\label{l:bound:measure}}
The measure $d\sigma_{\hat{\Delta}}$ is given by

\begin{equation*}
d \sigma_{\hat{\Delta}}|_{\{\langle p_j , x \rangle + c_j) L^j_{i}=0\}} = \frac{p}{\langle p_j , x \rangle + c_j} dx (d\sigma_{\Delta_B})|_{\{L^j_{i}=0\}} \quad \text{ if }  d_j \neq1,  
\end{equation*} 
and
\begin{equation*}
 d \sigma_{\hat{\Delta}}|_{\{\{\langle p_j , x \rangle + c_j)=0\}} = (pd\sigma_{\Delta})|_{\{L_j=0\}} dx_B  \quad \text{ if } d_j =1.
\end{equation*}
\end{lemma}
\begin{proof}
We fix a basis $(\xi_1,\dots,\xi_\ell,\xi^j_i)$, $j=0,\dots,\ell$, $i=1,\dots,d_{j-1}$, $d_j\neq1$, of the Lie algebra $\hat{\mathfrak{t}}=\mathfrak{t}\oplus \hat{\mathfrak{t}}_B$. After reordering the label $\mathbf{L}^j$, if necessary, we may assume that for $i=1, \dots, d_j - 1$, we have $L_i^j = x_i^j$ and $L_{0}^j = 1-\sum_{i=1}^{d_j-1} x_i^j$. 

Suppose first that $d_j \neq 1$. Hence the affine-linear function $\hat{L}_c$ is of the form $\hat{L}_c=(\langle p_j , x \rangle + c_j) L_i^j$. Suppose that $i=1,\dots,d_{j-1}$.  Hence $L_i^j=x^j_i$ and 
$\hat{L}_c=\hat{x}^j_i$. By restricting  \eqref{bound:mes} to the facet given by $\hat{L}_c=0$ and taking the interior product w.r.t. to $\xi^j_i=(d\hat{x}^j_i)^*$, we have the following equalities of volume forms
\begin{align*}
- d\sigma_{\{\hat{L}_c = 0\}}
=&  \xi^j_i \lrcorner (dx \wedge d\hat{x}_B)\\
=&  (-1)^\ell dx \wedge ( \xi^j_i \lrcorner d\hat{x}_B)\\
=&  (-1)^\ell dx \wedge (-1)^k \bigwedge_{\substack{0\le r\le \ell \\ 1\le q\le d_{r-1}\\ (r,q)\neq(i,j)}} d\hat{x}^r_q\\
=& (-1)^\ell  \frac{p}{\langle p_j , x \rangle + c_j}
\ dx \wedge (-1)^k \bigwedge_{\substack{0\le r\le \ell \\ 1\le q\le d_{r-1}\\ (r,q)\neq(i,j)}} dx^r_q\\
=& (-1)^\ell \frac{p}{\langle p_j , x \rangle + c_j}
\, dx \wedge (d\sigma_{\Delta_B})\big|_{\{L_i^j = 0\}},
\end{align*}
where  $k$ is a positive integer depending on $i$ and $j$. Hence as positive measures we obtain the desired assertion.

Next still assuming that $d_j\neq 1$, we consider $L_0^j= 1- \sum_{i=1}^{d_i-1}x^j_i$. Hence $\hat{L}_c=(\langle p_j , x \rangle + c_j) L_0^j=(\langle p_j , x \rangle + c_j) - \sum_{i=1}^{d_i-1}\hat{x}^j_i$. Let $u_j:=-\sum_{i=1}^{d_j-1}\xi_i^j$. Hence, similarly, we restrict \eqref{bound:mes} to the facet given by $\hat{L}_c=0$ and we take the interior product w.r.t. $d\hat{L}_c=p_j+u_j$:
\begin{equation*}
\begin{split}
- d\sigma_{\{\hat{L}_c = 0\}}
=&  (p_j \lrcorner dx ) \wedge  d\hat{x}_B + (-1)^\ell dx \wedge (u_j \lrcorner \hat{x}_B)\\
=&  - p d\sigma_{\Delta}|_{\{L_j=0\}} \wedge d\hat{x}_B - (-1)^\ell  \frac{p}{\langle p_j , x \rangle + c_j}
\, dx \wedge (d\sigma_{\Delta_B})\big|_{\{L_0^j = 0\}} \\
=&  - (-1)^\ell \frac{p}{\langle p_j , x \rangle + c_j}
\, dx \wedge (d\sigma_{\Delta_B})\big|_{\{L_0^j = 0\}},\\
\end{split}
\end{equation*}
since $p=0$ on $L_j$. Thus, the desired equality of positive measure follows.

\smallskip

The case $d_j=1$ is obtained by similar computations.

\end{proof}

\subsection{Scalar curvature and the extremal affine function}

We have the following key result giving the expression of the scalar curvature of a compatible metric and the affine extremal function.

\begin{lemma}{\label{l:scal}}
The  $v$-scalar curvature of a compatible metric $\omega$ on $X$ associated with a metric $\omega_{\mathbb{P}^\ell}$ on $\mathbb{P}^\ell$ is given by 
\begin{equation*}
     \Scal_v(\omega)= \Scal_{pv}(\omega_{\mathbb{P}^\ell}) +  \sum_{j=0 \atop
    d_j \neq 1}^\ell \frac{2d_j(d_j-1)}{\langle p_j, x \rangle + c_j},
\end{equation*}
where $2d_j(d_j-1)$ is the constant scalar curvature of $(\mathbb{P}(\mathbb{C}^{d_j}),\omega_{\mathrm{FS}_j})$. Moreover, the affine extremal function $\ell_{\mathrm{ext}}$ of $(X,[\omega],T)$ is the pullback of an affine function on $\Delta$ and $\langle p_j, x \rangle + c_j$ in introduce in Section \ref{s:set-up}.
\end{lemma}

This result is due to \cite[p. 380]{ACGT06} and \cite[Lemma 5]{ACGT11} when \( v = 1 \). The proof is written for semisimple principal fibrations in \cite[Section 5]{AJL} for any \( v > 0 \). The result in the previous reference is stated in the case where no blow-downs occurs, but the proof is local in nature and extends directly to our situation.

\subsection{Compatible test configurations}{\label{s:compatible:test}}
We consider the space $\mathcal{DPL}_{\mathrm{dom}}(\Delta)$ introduced in Section \ref{COnstruction:Donaldson}. 

\begin{prop}{\label{l:CT}}
Let $f \in \mathcal{DPL}_{\mathrm{dom}}(\Delta)$. Then $\pi^*(f)$ belongs to $\mathcal{DPL}_{\mathrm{dom}}(\hat{\Delta})$. In particular,  $f$ induces a smooth dominating toric test configuration $(\X_{f}, \A_{\X_f})$ for $(X,[\omega_{\mathrm{FS}}],T)$ defined as the toric manifold  associated with $\hat{\Delta}_{R-\pi^*f}$, see Definition \ref{d:torictc}.
\end{prop}

In what follows we will omit the pullback by $\pi$ in $\hat{ \Delta}_{R-\pi^*f}$ and simply write $\hat{ \Delta}_{R-f}$.

\begin{proof}
By definition, $\hat{\Delta}_{R-\pi^*f}$ is compact. Its simplicity follows from that of $\Delta_{R-f}$ and the simplex structure of $\hat{\Delta}$.

Since $\Delta_{R-f}$ is Delzant, each affine piece $f_k$ of $f=\max(f_1,\dots,f_r)$ can be written as
\[
f_k=\sum_{j=0}^{\ell}a_j^k L_j,
\qquad a_j^k\in\mathbb{Q}.
\]
Recall that $L_j=\langle p_j,x\rangle+c_j$. By Lemma \ref{l:label}, we have
\begin{align*}
f_k
&= \sum_{j=0}^{\ell} a_j^k L_j \\
&= \sum_{j=0}^{\ell} a_j^k
(\langle p_j,x\rangle+c_j)
\left(\sum_{i=0}^{d_j}L_i^j\right) \\
&= \sum_{j=0}^{\ell}\sum_{i=0}^{d_j}
a_j^k(\langle p_j,x\rangle+c_j)L_i^j.
\end{align*}

To pass from the first line to the second, we use the fact
that $\sum_{i=0}^{d_j}L_i^j=1$ for all
$j=0,\ldots,\ell$.
Hence, each $\pi^*f_k$ is a linear combination of the elements of $\hat{\mathbf L}$ with the same rational coefficients as $f_k$. In suitable integral affine coordinates, the primitive facet normals form a block-triangular matrix, whose nontrivial block corresponds to $\Delta_{R-f}$. Thus, the Delzant condition follows from that of $\Delta_{R-f}$. Moreover, the domination condition is preserved, so $\pi^*f\in\mathcal{DPL}_{\mathrm{dom}}(\hat{\Delta})$.
\end{proof}

\begin{define}{\label{d:compatible:tc}}
A smooth dominating toric test configuration $(\X_f, \A_{\X_f})$, constructed as in Proposition \ref{l:CT} for $(X,[\omega_{\mathrm{FS}}],\T)$, is called a compatible test configuration.

\end{define}

\subsection{Stability for compatible test configurations}

The following lemma expresses the reduced non-Archimedean $J$-functional of a compatible test configuration in terms of the weighted $J$-norm on $\Delta$.

\begin{lemma}{\label{l:J-fonc:compa:toric}}
Let $(\X_f, \A_{\X_f})$ be a compatible test configuration for $(X,[\omega_{\mathrm{FS}}],\T)$. Then

\begin{equation*}
\frac{1}{(2\pi)^{n+1} } \J_{\mathbb{T}^{\mathbb{C}}}(\X_f, \A_{\X_f})= \frac{\mathrm{Vol}(\Delta_B)}{\mathrm{Vol([\omega]})}   \inf_{\xi \in \mathrm{Aff}(\Delta)}\int_\Delta \big( f+\xi -\inf_\Delta (f+\xi)\big)pdx.
\end{equation*}

\end{lemma}

\begin{proof}
The proof follows from the computation in the toric case (Lemma \ref{l:J-fonc}), using \eqref{int:func}.
\end{proof}

For arbitrary weights $w$ on $X$, we introduce the corresponding weight $\tilde{w}$ on $\Delta$ by

\begin{equation}{\label{tildew}}
    \tilde{w}(x):=  w(x) - \sum_{j=0, d_j\neq1}^\ell \frac{2d_j(d_j-1)}{\langle p_j, x \rangle +c_j}
\end{equation}

\begin{lemma}{\label{fut-rigid}}
Let $(\X_f, \A_{\X_f})$ be a compatible test configuration. Then, the $(v,w)$-weighted Donaldson--Futaki invariant of $(\X_f, \A_{\X_f})$ is given by

\begin{equation}{\label{obj1}}
    \mathbf{DF}_{v,w}(\X_f, \A_{\X_f})= (2\pi)^{n+1}\mathrm{Vol}(\Delta_B)\mathcal{F}^{\Delta}_{pv,\tilde{w}}(f),
\end{equation}

\noindent . In particular,

\begin{equation}{\label{equa:fut}}
\mathcal{F}^{\hat{\Delta}}_{v,w}(f)=\mathrm{Vol}(\Delta_B)\mathcal{F}^{\Delta}_{pv,\tilde{w}}(f).
\end{equation}

\end{lemma}

\begin{proof}
By Lemma \ref{l:toric:fut},

\begin{equation*}
    \frac{1}{(2\pi)^{n+1}}\mathbf{DF}_{v,w}(\mathcal{X}_f,\A_{\X_f})= \mathcal{F}^{\hat{\Delta}}_{v,w}(f).
\end{equation*}

We first compute the interior term of $\mathcal{F}^{\hat{\Delta}}_{v,w}$. From \eqref{int:func}, we obtain

\begin{equation*}
    \int_{\hat{\Delta}} fwv\,d\hat{x} = \mathrm{Vol}(\Delta_B) \int_{\Delta} fwvp\,dx.
\end{equation*}

For the boundary term, Lemma \ref{l:bound:measure} gives

\begin{equation*}
\begin{split}
    2\int_{\partial \hat{\Delta}} fv d \sigma_{\hat{\Delta}}  = &2\sum_{j=0 \atop d_j\neq 1}^\ell \frac{\mathrm{Vol}(\Delta_B)}{\mathrm{Vol}(\Delta_j)} \int_{\partial \Delta_j} d\sigma_{\Delta_j} \int_{\Delta} \frac{1}{\langle p_j,x\rangle +c_j} f pv dx \\
    &+ 2\mathrm{Vol}(\Delta_B)\int_{\partial \Delta} fpv d \sigma_{\Delta}.
\end{split}
\end{equation*}

Since each $\Delta_j$ is the Delzant polytope of the Fubini--Study metric on $\mathbb{P}(\mathbb{C}^{d_j})$, integration by parts (see \cite[Lemma 3.3.5]{SKD} or \cite[Proposition 3.1]{Apo}) yields

\begin{equation*}
   2\int_{\partial \Delta_j} d\sigma_{\Delta_j}
   = 2d_j(d_j-1)\mathrm{Vol}(\Delta_j).
\end{equation*}

Combining these identities gives \eqref{equa:fut}, and hence \eqref{obj1}.
\end{proof}

The same computation applies to any continuous function $f$ on $\Delta$, viewed as a function on $\hat{\Delta}$ through the natural projection. Thus,

\begin{equation}{\label{equa:fut2}}
\mathcal{F}^{\hat{\Delta}}_{v,w}(f)=\mathrm{Vol}(\Delta_B)\mathcal{F}^{\Delta}_{pv,\tilde{w}}(f).
\end{equation}

We now establish the analogue of Lemma \ref{l:smooth:test} for the weights $(pv,\tilde{w})$, noting that $p$ may vanish on $\partial\Delta$.

\begin{lemma}{\label{p:stab:eq:rigid}}
The following statements are equivalent.

\begin{enumerate}
    \item The Delzant polytope $\Delta$ is $(pv,\tilde{w})$-weighted uniformly K-stable, in the sense of Definition \ref{d:stability}, with the Futaki functional defined by the product $pv\tilde{w}$.
    \item There exists $\lambda > 0$ such that for all $f \in \mathcal{DPL}_{\mathrm{dom}}(\Delta)$,

\begin{equation*}
\mathcal{F}^{\Delta}_{pv,\tilde{w}}(f) \geq \lambda \lVert f\rVert_{J_{pv}(\Delta)}.
\end{equation*}
\end{enumerate}
\end{lemma}

\begin{proof}
The result follows directly from Lemma \ref{l:smooth:test} in Appendix \ref{a:toric}. Indeed, although $p$ may vanish on the boundary of $\Delta$, the weight $pv$ is non-negative on $\Delta$ and strictly positive in its interior. Moreover, both $pv$ and $pv\tilde{w}$ extend smoothly to $\Delta$. Thus, the weights satisfy the assumptions of Lemma \ref{l:smooth:test}, which gives the desired equivalence.
\end{proof}

We conclude this section with the following characterization of weighted stability for compatible test configurations.

\begin{prop}{\label{l:kstab:stab:delta}}
Let $X=\mathbb{P}(\oplus_{j=0}^\ell\mathbb{C}^{d_j})$ be the projective space endowed with the Fubini--Study metric $\omega_{\mathrm{FS}}$. Let $\Delta$ be the $\ell$-simplex associated with the $\T$-action. Then $(X,[\omega_{\mathrm{FS}}],\T)$ is $(v,w)$-uniformly K-stable for compatible test configurations if and only if $\Delta$ is $(pv,\tilde{w})$-uniformly K-stable.
\end{prop}

\begin{proof}
This follows directly from Lemmas \ref{l:J-fonc:compa:toric}, \ref{fut-rigid}, and \ref{p:stab:eq:rigid}.
\end{proof}

\section{Existence of a compatible weighted cscK metric on the projective space}{\label{s:exi:x}}

We consider 
\(
X = \mathbb{P}\!\left(\oplus_{j=0}^{\ell} \mathbb{C}^{\,d_j}\right)
\)
as described in Section~\ref{s:set-up}, equipped with the Fubini--Study metric $\omega_{\mathrm{FS}}$. 
The aim of this section is to prove Theorem \ref{t:A}. We first establish a slightly more general result in Theorem \ref{t:exi:x} below, allowing for arbitrary weights $v>0$ and $w$. We then recover the desired statement by taking $v=\bar{p}$ and $w=\hat{w}$, where $\hat{w}$ is the smooth fiber weight introduced in Example~\ref{e:proj:bund} and its associated weight on $\Delta$ is $\tilde{w}=\bar{w}$.

\smallskip

For the arbitrary weights $(v,w)$ on $X$, let $\tilde{w}$ be the reduced weight defined in \eqref{tildew}. We introduce the $(pv,\tilde{w})$-weighted Mabuchi energy of $(\Delta, \mathbf{L})$, defined for a symplectic potential $u \in \mathcal{S}(\Delta, \mathbf{L})$, by
\[
\mathcal{M}_{pv,\tilde{w}}^{\Delta}(u)
= -\int_{\Delta} \log \det \bigl(\mathrm{Hess}(u) \, \mathrm{Hess}(u_0)^{-1}\bigr) \, pv \, dx 
+ \mathcal{F}^{\Delta}_{pv,\tilde{w}}(u),
\]
where $\tilde{w}$ is defined by \eqref{tildew}, $p$ by \eqref{weight:x}, and $u_0$ is the Guillemin potential of $\Delta$ \cite{Gui}.

We consider the compatible symplectic potential $\hat{u}_0 \in \mathcal{S}(\hat{\Delta},\hat{\mathbf{L}}) $ constructed from $u_0 \in \mathcal{S}(\Delta,\mathbf{L})$  as in \eqref{comp:symp}. Similarly, we define the $(v,w)$-weighted Mabuchi energy $\mathcal{M}_{v,w}^{\hat{\Delta}}$ of $(\hat{\Delta}, \hat{\mathbf{L}})$ with reference potential $\hat{u}_0$

\[
\mathcal{M}_{v,w}^{\hat{\Delta}}(\hat{u})
= -\int_{\hat{\Delta}} \log \det \bigl(\mathrm{Hess}(\hat{u}) \, \mathrm{Hess}(\hat{u}_0)^{-1}\bigr) \, v \, dx 
+ \mathcal{F}^{\hat{\Delta}}_{v,w}(\hat{u})
\]
for any $\hat{u} \in \mathcal{S}(\hat{\Delta},\hat{\mathbf{L}})$.

The relation between these functionals is given by the following Lemma.

\begin{lemma}{\label{l:mab:del:mab:hat}}
Let $\hat{u} \in \mathcal{S}(\hat{\Delta},\hat{\mathbf{L}})$ be a compatible symplectic potential associated with a symplectic potential $u \in \mathcal{S}(\Delta, \mathbf{L})$ as in Lemma \ref{Lemma-symp-pot}. Then
\begin{equation*}
\begin{split}
  \mathcal{M}^{\hat{\Delta}}_{v,w}( \hat{u}) = &\mathrm{Vol}(\Delta_B) \mathcal{M}^{\Delta}_{pv,\tilde{w}}(u) + C,
\end{split}  
\end{equation*}
where $C:=\mathcal{F}_{v,w}^{\hat{\Delta}}\left(\sum_{j=0 \atop d_j \neq 1 }^\ell( \langle p_j,x\rangle + c_j)u_j)\right)$ is independent of $u$.
\end{lemma}

\begin{proof}
By Lemma \ref{Lemma-symp-pot} we get

\begin{equation}{\label{det:hu}}
    p\det \textnormal{Hess}(\hat{u})= \det \textnormal{Hess}(u)  \det \textnormal{Hess}(u_B),
\end{equation}
where $u_B:= \sum_{j=0 \atop d_j\neq 1}^\ell u_j$. Then

\begin{equation*}
    \frac{\det \textnormal{Hess}(\hat{u})}{\det \textnormal{Hess}(\hat{u}_0)} = \frac{\det \textnormal{Hess}(u)}{\det \textnormal{Hess}(u_0)}.
\end{equation*}

Hence by \eqref{int:func} we deduce that

\begin{equation}{\label{ent}}
    \begin{split}
    \int_{\hat{\Delta}}\log \det \textnormal{Hess}(\hat{u}) \det \textnormal{Hess}(\hat{u}_0)^{-1} v d\hat{x}  
      = \mathrm{Vol}(\Delta_B)\int_{\Delta} \log \det \textnormal{Hess}(u)\textnormal{Hess}(u_0)^{-1} pvdx 
    \end{split}
\end{equation}
The computation of the linear part follows from Lemma \ref{Lemma-symp-pot} and \eqref{equa:fut2}.
\end{proof}

\begin{lemma}{\label{stable-equivaut-energy-propr-v}}
Suppose that $\Delta$ is $(pv,\tilde{w})$-uniformly K-stable. Then there exist constants  $C>0$ and $D >0 $ such that
    
\begin{equation}{\label{coerc:toric}}
    \mathcal{M}_{pv,\tilde{w}}^{\Delta}(u) \geq C \|u^*   \|_{L^1_{pv}(\Delta)} - D
\end{equation}

\noindent for all $u \in \mathcal{S}(\Delta,\mathbf{L})$.
\end{lemma}

This result for $p=v=1$ is due to \cite{SKD,ZZ} and was generalized in \cite[Proposition~7.9]{Jub23} to the case $v>0$ and $p=1$.

\begin{proof}

The main difficulty is that \(p\) may vanish on the boundary of \(\Delta\). To circumvent this issue, we consider the projection \(\pi : \hat{\Delta} \to \Delta\) and work instead on \(\hat{\Delta}\).

By assumption, \(\Delta\) is \((pv,\tilde{w})\)-uniformly K-stable. Therefore, by \eqref{int:func}, \eqref{equa:fut2}, and Lemma \ref{l:stab:eq}, there exists \(\lambda>0\) such that for any \(f \in \mathcal{CV}_*^{\infty}(\Delta)\),
\begin{equation}\label{stab}
    \mathcal{F}^{\hat{\Delta}}_{v,w}(f) \geq \lambda \| f \|_{L^1_v(\hat{\Delta})}.
\end{equation}
Let \(\hat{u}\) denote the symplectic potential in \(S(\hat{\Delta},\hat{\mathbf{L}})\) associated with a symplectic potential \(u \in S(\Delta,\mathbf{L})\) via Lemma \ref{Lemma-symp-pot}. Then
\begin{equation*}
\begin{split}
\mathcal{M}^{\hat{\Delta}}_{v,w}(\hat{u})
= &\, \mathcal{F}^{\hat{\Delta}}_{v,w}(u)
 - \int_{\hat{\Delta}} \log \det \bigl(\mathrm{Hess}(\hat{u})\, \mathrm{Hess}(\hat{u}_0)^{-1}\bigr)\, v \, dx + C,
\end{split}
\end{equation*}
where
\(
C := \mathcal{F}_{v,w}^{\hat{\Delta}}\!\left(\sum_{j=0 \atop d_j \neq 1}^\ell (\langle p_j,x\rangle + c_j)u_j\right)
\)
is independent of \(u\).

Since \(v>0\) and \(u \in \mathcal{CV}^{\infty}(\Delta)\), we may apply \eqref{stab} together with the arguments of \cite[Proposition 7.9]{Jub23} to obtain constants \(C_1, C_2, C_3>0\), independent of \(u\), such that
\begin{equation*}
\begin{split}
\mathcal{M}^{\hat{\Delta}}_{v,w}(\hat{u}) \geq &\, -C_3+ C_1 \|u^*\|_{L_v^1(\hat{\Delta})}
 - \int_{\hat{\Delta}} \log \det \bigl(\mathrm{Hess}(\hat{u})\, \mathrm{Hess}(\hat{u}_0)^{-1}\bigr)\, v \, dx \\
& + \int_{\hat{\Delta}} \log \frac{\det \mathrm{Hess}(C_2u)}{\det \mathrm{Hess}(u_0)}\, v \, dx .
\end{split}
\end{equation*}
The determinant identity \eqref{det:hu}  gives an additive constant independent of $u$. 

Therefore, there exists a constant \(D>0\), independent of \(u\), such that
\begin{equation*}
    \mathcal{M}^{\hat{\Delta}}_{v,w}(\hat{u}) \geq D \|u^*\|_{L_v^1(\hat{\Delta})} - D.
\end{equation*}
We conclude by applying Lemma \ref{l:mab:del:mab:hat} and \eqref{int:func}.

\end{proof}

\begin{theorem}{\label{t:exi:x}}
 If $\Delta$ is $(pv,\tilde{w})$-uniformly K-stable, there exists a compatible $(v,w)$-cscK metric in $(X,[\omega_{\mathrm{FS}}])$.
\end{theorem}

\begin{proof}
The ideas are similar to those in \cite[Theorem~7.12]{Jub23}, so we only sketch the argument.

\smallskip

Recall that the Fubini--Study metric on $X$ is a compatible K\"ahler metric in the sense of Definition~\ref{d:comp:metric}. Let $\omega_{\P^\ell}$ denote the K\"ahler metric on $\P^\ell$ inducing $\omega_{\mathrm{FS}}$ on $X$.

By Lemma~\ref{stable-equivaut-energy-propr-v}, the functional $\mathcal{M}^{\Delta}_{pv,\tilde{w}}$ is coercive, i.e. it satisfies \eqref{coerc:toric}. Using the same argument as in \cite[Eqs.~(63) and (64)]{Jub23} together with Lemma~\ref{l:stab:eq}, the weighted Mabuchi energy $\mathbf{M}_{pv,\tilde{w}}^{\P^\ell}$ of $(\P^\ell,[\omega_{\P^\ell}])$ is $\T^{\C}$-coercive with respect to the $p$-weighted $\mathbf{J}$-functional $\mathbf{J}^{\P^\ell}_{p}$ with reference metric $\omega_{\P^\ell}$ (see e.g. \cite[Definition~6.2 and Remark~6.3]{AJL}). More precisely, there exist constants $C,D>0$ such that for any $\omega \in \mathcal{K}(\P^\ell)^\T$,
\begin{equation}{\label{mab:coerc}}
 \mathbf{M}_{pv,\tilde{w}}^{\P^\ell}(\omega)
 \geq C \inf_{\gamma \in \T^{\C}} \mathbf{J}^{\P^\ell}_{p}(\gamma^*\omega) - D,
\end{equation}
where  $\mathcal{K}(\P^\ell)^\T$ denotes the space of $\T$-invariant K\"ahler metrics in $[\omega_{\P^\ell}]$. We cannot apply the results of \cite{DJL24, DJL25} to obtain the existence of a $(pv,\tilde{w})$-cscK metric here, since the weight $p$ may vanish. Hence, we will work on $X$.

We now consider the embedding of $\mathcal{K}(\P^\ell)^\T$ into $\mathcal{K}(X)^\T$, the space of $\T$-invariant K\"ahler metrics in $[\omega_{\mathrm{FS}}]$, described in Section~\ref{s:set-up}. Its image consists precisely of the compatible K\"ahler metrics. 

Let $\mathbf{M}_{v,w}^{X}$ be the $(v,w)$-weighted Mabuchi energy of $(X,[\omega_{\mathrm{FS}}])$. Then, by Lemma~\ref{l:scal} and \cite[(8)]{ACGT11},

\begin{equation}{\label{restric:mab}}
\mathbf{M}_{v,w}^{X}\big|_{\mathcal{K}(\P^\ell)^\T}
=\mathrm{Vol}(B,\omega_B)\mathbf{M}_{pv,\tilde{w}}^{\P^\ell}.
\end{equation}
where $\mathrm{Vol}(B,\omega_B)$ denotes the volume of $(B,\omega_B)$.

Let $\mathbf{J}^X:=\mathbf{J}^X_1$ denote the usual unweighted $\mathbf{J}$-functional on $\mathcal{K}(X)$ with reference metric $\omega_{\mathrm{FS}}$ (see e.g. \cite[Definition~3.1]{AJL}). Using \cite[(8)]{ACGT11}, \cite[Lemma~7]{ACGT11}, and the second equality in the proof of \cite[Lemma~6.4]{AJL}, we obtain
\begin{equation}{\label{restric:d1}}
   \mathbf{J}^X\big|_{\mathcal{K}(\P^\ell)^\T}
   =
   \mathrm{Vol}(B,\omega_B)\,\mathbf{J}_{p}^{\P^\ell}.
\end{equation}

Combining \eqref{mab:coerc}, \eqref{restric:mab}, and \eqref{restric:d1}, we deduce that the restriction of $\mathbf{M}_{v,w}^{X}$ to the space of compatible K\"ahler metrics is $\T^{\C}$-coercive. More precisely, there exist constants $C',D'>0$ such that for any compatible metric $\tilde{\omega} \in \mathcal{K}(\P^\ell)^\T \subset \mathcal{K}(X)^\T$,
\begin{equation}{\label{mab:coerc2}}
 \mathbf{M}^X_{v,w}(\tilde{\omega})
 \geq C' \inf_{\gamma \in \T^{\C}}\mathbf{J}^X(\gamma^*\tilde{\omega}) - D'.
\end{equation}

\smallskip

We now consider the $(v,w)$-weighted continuity path of Chen on $\mathcal{K}(X)^\T$ as in \cite[(21)]{DJL25}:
\begin{equation}{\label{cont:path1}}
  t\big(\mathrm{Scal}_{v}(\tilde{\omega})-w(\mu_{\tilde{\omega}})\big)
  =
  (1-t)\big(\Lambda_{\tilde{\omega},v}(\tilde{\omega}_0)-\tilde{v}(\mu_{\tilde{\omega}})\big),
  \qquad t\in[0,1],  \tilde{\omega} \in \mathcal{K}(X)^\T,
\end{equation}
where $\tilde{v}:= n + \langle d\log(v), x \rangle$ is a fixed weight with $x \in \Delta$, $\Lambda_{\tilde{\omega},v}$ denotes the $v$-weighted trace defined in Section~\ref{s:reviewcscK} and $\tilde{\omega}_0$ is a suitable compatible K\"ahler metric on $X$ that will be chosen below.

In \cite[Theorem 1.1]{DJL25}, it is shown that this continuity path can be solved at $t=1$ whenever $\mathbf{M}^X_{v,w}$ is $\T^{\C}$-coercive on $\mathcal{K}(X)^\T$. In our setting, however, we only know coercivity on the subset of compatible metrics $\mathcal{K}(\P^\ell)^\T \subset \mathcal{K}(X)^\T$. We therefore follow the strategy of \cite[Theorem 6.1]{Jub23} (see \cite[Corollary~5.12]{DJL25} for the weighted generalization) and solve the continuity path within this smaller space under this weaker hypothesis. 

More precisely, set
\[
    S:=\{ t \in (0,1] \mid \eqref{cont:path1} \text{ admits a compatible solution } \tilde{\omega} \in \mathcal{K}(\P^\ell)^\T \subset \mathcal{K}(X)^\T \}.
\]
It is enough to show that $S$ is open, closed and non-empty.

Openness of $S$ follows from a careful application of the implicit function theorem, as in \cite[Proposition~6.2]{Jub23}, using the decomposition of the linearized operator of \eqref{cont:path1} into vertical and horizontal components established in \cite[Lemma~5.11]{DJL25} instead of \cite[(43)]{Jub23}.

For the existence of a solution for $t \in (0,1)$, we follow \cite[Lemma~6.3]{Jub23}, choosing $\tilde{\omega}_0$ carefully as in \cite[(40)]{Jub23}.

Closedness follows as in \cite[Proposition~6.5]{Jub23}, using \eqref{mab:coerc2} and the existence result for weighted cscK metrics from \cite[Theorem~1.1]{DJL25} and \cite[Theorem B]{JubertPan26}, in place of the existence of extremal metric results of \cite[Theorem~1.6]{CC21c} and \cite[Theorem~3.3]{He}.

We therefore obtain a compatible solution of \eqref{cont:path1} at $t=1$, which concludes the proof.

\end{proof}

We can now complete the proof of Theorem \ref{t:A}.

\begin{proof}
Consider 
\begin{equation*}
Y=\mathbb{P}\left(\oplus_{j=0}^{\ell}E_j\right)\longrightarrow C.   
\end{equation*}
with fiber $X$.

We recall hat the weight associate to this fibration are introduced in Example \ref{e:proj:bund}.

Fix a K\"ahler class $[\omega_Y]$ on $Y$. By \cite[Lemma~3.3]{AK}, every K\"ahler class on $Y$ admits, up to scaling, a compatible representative. We may therefore assume that $\omega_Y$ is a compatible K\"ahler metric induced by a compatible K\"ahler metric $\omega_X$ in $[\omega_{\mathrm{FS}}]$ on the fiber $X$, see Section \ref{s:comp:bundle}.

By the generalized Calabi ansatz of \cite{ACGT06, ACGT11} (see also Lemma \ref{l:scal}), any compatible $(\bar{p},\hat{w})$-weighted cscK metric $\omega_X$ in $[\omega_{\mathrm{FS}}]$ induces an extremal K\"ahler metric on $Y$ in the prescribed K\"ahler class.

The hypotheses of Theorem \ref{t:exi:x} are satisfied by the geometric weights. Indeed, $\bar{p}$ is strictly positive and log-concave on $\Delta$, since the prescribed class is K\"ahler, while $p$ is non-negative and positive in the interior. Moreover, although $\bar{w}$ contains reciprocal facet terms, $p\bar{w}$ extends smoothly to $\Delta$. Thus, Theorem \ref{t:exi:x} applies to the weights $(\bar{p},\hat{w})$ on $X$, corresponding to $(p\bar{p},\bar{w})$ on $\mathbb{P}^\ell$, and yields the desired existence result.

Moreover, the exact same argument than in \cite[Corollary 6.6]{Jub23} shows that any extremal K\"ahler metric is compatible.

\medskip

The converse direction follows from \cite[Theorem A]{Jub26}.

\end{proof}

\begin{corollary}
Let $[\omega_Y]$ be a polarized K\"ahler class on $Y$, i.e., $[\omega_Y]=2\pi c_1(L_Y)$ for some ample line bundle $L_Y$. Then the following statements are equivalent.
\begin{enumerate}
    \item There exists an extremal metric in $(Y,[\omega_Y])$.
    \item There exists a $(\bar{p},\hat{w})$-weighted cscK metric in $(X,[\omega_{\mathrm{FS}}])$.
    \item $\Delta$ is $(p\bar{p},\bar{w})$-uniformly $K$-stable.
\end{enumerate}
\end{corollary}

\begin{proof}
The equivalence $(1)\Leftrightarrow(3)$ follows directly from Theorem \ref{t:A}.

Assume (3). Applying Theorem \ref{t:exi:x} with $v=\bar{p}$ and $w=\hat{w}$, we obtain a compatible $(\bar{p},\hat{w})$-weighted cscK metric on $(X,[\omega_{\mathrm{FS}}])$, proving (2).

Conversely, suppose that (2) holds. By \cite[Corollary 1]{AJL}, the existence of a $(\bar{p},\hat{w})$-weighted cscK metric implies weighted uniform $K$-stability with respect to compatible test configurations. Proposition \ref{l:kstab:stab:delta} then implies that $\Delta$ is $(p\bar{p},\bar{w})$-uniformly $K$-stable, proving (3).
\end{proof}

\section{Horospherical symmetries of compatible test configurations of projective space}\label{horo tc for proj}
\subsection{Projective space of direct sums and polytopes}
Let $X = \mathbb{P}\left(\oplus_{j=0}^\ell \mathbb{C}^{d_j}\right)$, $G = \prod_{j=0}^\ell U(d_{j})$ and $G^{\mathbb{C}} = \prod_{j=0}^\ell \mathrm{GL}(d_{j}, \mathbb{C}).$ There is a natural \( G^{\mathbb{C}}\)-action on \( X \). For an element in $X$, we write it as 
\[
[\mathbf{z}_0 : \mathbf{z}_1 : \dots : \mathbf{z}_\ell],
\]
where $\mathbf{z}_{j} \in \mathbb{C}^{d_j}$, and $\mathbf{z}_j$ can be written as 
\[
\mathbf{z}_j = (z_{j,0}, z_{j,1}, \dots, z_{j,d_{j}-1}).
\]

Let $\mathbf{e_j} = (1, 0, \dots, 0)\in \mathbb{C}^{d_{j}}$.
  
Let us fix the Borel subgroup of $G^{\mathbb{C}}$ as \( \prod_{i=0}^\ell B_i \), where \( B_i \) is the subgroup of upper triangular matrices in $\mathrm{GL}(d_{i},\mathbb{C})$. We fix the maximal algebraic torus of $G^{\mathbb{C}}$ as the product of diagonal matrices.

\begin{lemma}
Under the $G^{\mathbb{C}}$-action, \( X \) is a horospherical variety. 
\end{lemma}

\begin{proof}
To see this, let us consider the point \( x = [\mathbf{e}_0 : \mathbf{e}_1 : \mathbf{e}_2 : \dots : \mathbf{e}_\ell] \). Clearly, \( G^{\mathbb{C}} \cdot x \) is open in \( X \) and  \( G^{\mathbb{C}} \cdot x \subset X\) is a $G^{\mathbb{C}}$-equivariant embedding. We denote $G^{\mathbb{C}} \cdot x$ as $\mathring{X}$, we have
\[\mathring{X} = \{[\mathbf{z}_0 : \mathbf{z}_1 : \dots : \mathbf{z}_\ell]\,|\, \mathbf{z}_{j}\neq 0\}.\]
The subgroup of $G^{\mathbb{C}}$ which fixes \( x \) is 
\[
H = \left\{ \left(
\begin{pmatrix}
a_0 & B_{0} \\
0 & C_0
\end{pmatrix},
\begin{pmatrix}
a_1 & B_{1} \\
0 & C_1
\end{pmatrix},\cdots,
\begin{pmatrix}
a_{l} & B_{l} \\
0 & C_\ell
\end{pmatrix}\right)
\mid a_i \in \mathbb{C}^*, a_0 = a_1 = \dots = a_\ell 
\right\}.
\]
The unipotent radical of $B$ is contained in $H$, thus $H$ is a horospherical subgroup (see section \ref{preliminarieshoro} ) and $X$ is a $G^{\mathbb{C}}$-horospherical variety.
\end{proof}

We denote the algebraic moment polytope of $\mathcal{O}_{X}(1)$ with respect to $G^{\mathbb{C}}$ as $\Delta_{\mathrm{alg}}$. Note that, $z_{i,d_{i}}$ considered as elements in $H^{0}(X,\mathcal{O}(1))$, are $B$-weight vectors. We denote the weight of $z_{i,d_{i}}$ as $\chi_{i}$. The algebraic moment polytope $\Delta_{\mathrm{alg}}$ is exactly generated by $\chi_{i}$. We also have the polytope $\Delta = \Delta_{\mathrm{alg}}-\chi_{0}$ which is the polytope related to the section $z_{0,d_{0}}.$

The spherical lattice $M$ of $X$ is generated by $\frac{z_{i,d_{i}}}{z_{0,d_{0}}}$ for $i\neq 0$. We thus have
\[M = \mathbb{Z}\text{-}\{\chi_{i}-\chi_{0} \,|\,i\neq 0\} \cong \mathbb{Z}^\ell\]

Now we consider the symplectic side of the story.

We consider the Kirwan polytope $\Delta_{\mathrm{Kir}}$ of $(X,\omega_{\mathrm{FS}})$ under the action of $G = \prod_{j=0}^{\ell}U(d_{j})$. The moment map of $(X,\omega_{\mathrm{FS}})$ with respect to $G$ is
\[
\mu^{G}_{\omega_{\mathrm{FS}}}([\mathbf{z}_0 : \mathbf{z}_1 : \dots : \mathbf{z}_\ell])(a_{0},\cdots, a_{l}) = \frac{\sum_{j=0}^{\ell}\mathbf{z}_j^* a_{j}\mathbf{z}_j}{i\sum_{j=0}^{\ell} |\mathbf{z}_j|^2} ,
\]
where $a_j \in \mathfrak{u}(d_j)$. In particular,
\[
\mu^{G}_{\omega_{\mathrm{FS}}}\left([0: \dots : 0 : \mathbf{e_j} : 0 : \dots : 0]\right)(a_0, a_1, \dots, a_n) = \frac{a_{j,11}}{i}.
\]
The Kirwan polytope $\Delta_{\mathrm{Kir}}$ is exactly the polytope generated by all
\[
\mu^{G}_{\omega_{\mathrm{FS}}}\left([0: \dots : 0 : \mathbf{e_j} : 0 : \dots : 0]\right).
\]
We also have a natural $\mathbb{T} := (\mathbb{S}^1)^{\ell}$-action on $X$:
\[
(e^{it_1}, \dots, e^{it_\ell}) \cdot [\mathbf{z}_0 : \dots : \mathbf{z}_\ell] = 
[\mathbf{z}_0 : e^{it_1}\mathbf{z}_1 : \dots : e^{it_\ell}\mathbf{z}_\ell].
\]
The moment map of $\omega_{\mathrm{FS}}$ with respect to the $\mathbb{T}$-action is,
\[
\mu_{\omega_{\mathrm{FS}}}^{\mathbb{T}}([\mathbf{z}_0 : \mathbf{z}_1 : \dots : \mathbf{z}_\ell]) (it_1, \dots, it_\ell) =
\frac{t_1 |\mathbf{z}_1|^2 + \dots + t_\ell |\mathbf{z}_\ell|^2}{|\mathbf{z}_0|^2 + |\mathbf{z}_1|^2 + \dots + |\mathbf{z}_\ell|^2}.
\]

The torus $\mathbb{T}$ can be considered as a subgroup of $G$. So $\mu_{\omega_{\mathrm{FS}}}^{\mathbb{T}}$ is simply the composition of $\mu^{G}_{\omega_{\mathrm{FS}}}$ and the projection $\mathfrak{g}^{*}\rightarrow \mathfrak{t}^{*},$ where $\mathfrak{g}$ is the Lie algebra of $G$ and $\mathfrak{g}^*$ is the dual of $\mathfrak{g}.$ Since $\mathbb{T}$ is contained in the center of $G$, we have a map $\Delta_{\mathrm{Kir}}\rightarrow \Delta,$ where $\Delta$ is the symplectic moment polytope. A direct computation shows that:

\begin{lemma}
The natural map between the two polytopes $\Delta_{\mathrm{Kir}}$ and $\Delta$ is an isomorphism.
\end{lemma}

Now assume that we have a $f \in \mathcal{DPL}_{\mathrm{dom}}(\Delta)$. We choose a rational $R>0$ such that $R-f$ is strictly positive. We then have the polytope:
\[\Delta_{R-f} = \{(x,y) \, | \, x
\in \Delta,0\leq y\leq R-f\}\]
The function $R-f$ induces a function on $\Delta_{\mathrm{Kir}}.$ We then have the polytope $\Delta_{R-f,\mathrm{Kir}}$ defined in the same way as $\Delta_{R-f}.$

This provides a $G^{\mathbb{C}}$-equivariant test configuration $(\mathcal{X}_{f},\mathcal{L}_{f})$ of $(X,\mathcal{O}(1))$ following Theorem \ref{kir_alg_id} and Theorem \ref{sphericaltest}. This $\mathcal{X}_{f}$ is a $\Gc\times\mathbb{C}^{*}$-equivariant embedding of $\mathring{\mathcal{X}_{f}}:=\mathring{X}\times \mathbb{C}^{*}\cong \Gc\times\mathbb{C}^{*} /(H\times\{1\} )$, so it is a $\Gc\times\mathbb{C}^{*}$-horospherical variety. We conclude as follows.

\begin{lemma}
For $f \in \mathcal{DPL}_{\mathrm{dom}}(\Delta)$, we have a corresponding $G^{\mathbb{C}}$-equivariant test configuration $(\mathcal{X}_{f},\mathcal{L}_{f})$ of $(X,\mathcal{O}(1))$. The variety $\mathcal{X}_{f}$ is a $\Gc\times\mathbb{C}^{*}$-horospherical variety.
\end{lemma}

\subsection{Smoothness of the horospherical test configuration}\label{smoothnesstc}
Choose $k$ large enough that $k\mathcal{L}_{f}$ gives an equivariant projective embedding. Fix a $G\times\mathbb{S}^{1}$-invariant Hermitian inner product on the representation $H^{0}(\mathcal{X}_{f},k\mathcal{L}_{f})^{*}$ (for instance, by averaging any Hermitian product over the compact group), and use this inner product to define the corresponding Fubini--Study metric. Pulling back its suitably normalized form under the equivariant embedding $\mathcal{X}_{f}\hookrightarrow\mathbb{P}(H^{0}(\mathcal{X}_{f},k\mathcal{L}_{f})^{*})$ gives a $G\times\mathbb{S}^{1}$-invariant K\"ahler form $\omega_{\mathcal{X}_{f}}$ representing $2\pi c_1(\mathcal{L}_{f})$. Let $\mu_{\omega_{\mathcal{X}_{f}}}^{\tl\times \mathbb{S}^{1}}$ (respectively $\mu_{\omega_{\mathcal{X}_{f}}}^{G\times \mathbb{S}^{1}}$) be the moment map of $(\mathcal{X}_{f},\omega_{\mathcal{X}_{f}})$ under the action of $\mathbb{T}\times \mathbb{S}^{1}$ (respectively $G\times \mathbb{S}^{1}$).

The polytope $\Delta_{R-f,\mathrm{Kir}}$ is exactly the Kirwan polytope of $(\mathcal{X}_{f},\omega_{\mathcal{X}_{f}})$ under the $G\times \mathbb{S}^{1}$-action. This follows from Theorem~\ref{kir_alg_id} and the fact that 
$\Delta_{R-f,\mathrm{alg}}$ is precisely the algebraic moment polytope of 
$(\mathcal{X}_{f},\mathcal{L}_{f})$ under the $G^{\mathbb{C}}\times\mathbb{C}^{*}$–action, 
as explained in the proof of \cite[Theorem~4.1]{Del2023}.

We have a standard $T$-action on $X$ which makes $(X,\omega_{\mathrm{FS}})$ a Kähler toric manifold. We have a corresponding $T$-action on $\mathcal{X}_{f}$ since $T$ is a subgroup of $G.$

Note that the image of $\mu_{\omega_{\mathcal{X}_{f}}}^{G\times \mathbb{S}^{1}}$ is the union of all coadjoint orbits of points in $\Delta_{R-f,\mathrm{Kir}}$ under the $G\times \mathbb{S}^{1}$-action. It follows directly that the symplectic moment polytope of $(\mathcal{X}_{f},\omega_{\mathcal{X}_{f}})$ under $T\times \mathbb{S}^{1}$-action is $\hat{\Delta}_{R-f}$. From Proposition \ref{l:CT} we know that $\hat{\Delta}_{R-f}$ is a Delzant polytope. 

Note that $\mathcal{X}_{f}$ is a $T^{\mathbb{C}}\times \mathbb{C}^{*}$-equivariant embedding of $T^{\mathbb{C}}\times \mathbb{C}^{*}$. We denote the Kähler class of $2\pi c_{1}(\mathcal{L}_{f})$ as $\mathcal{A}_{\mathcal{X}_{f}}.$ Then we have:

\begin{prop}
    The corresponding horospherical test configuration $\mathcal{X}_{f}$ is smooth. $(\mathcal{X}_{f},\mathcal{A}_{\mathcal{X}_{f}})$ is a Kähler $T\times \mathbb{S}^{1}$-toric manifold whose symplectic moment polytope is $\hat{\Delta}_{R-f}$. In particular, as a toric manifold, it is isomorphic to compatible test configurations introduced in Definition \ref{d:compatible:tc}.
\end{prop}

\subsection{Semisimplicity and rigidity of the $\mathbb{T}\times \mathbb{S}^{1}$-action on $(\mathcal{X}_{f},\omega_{\mathcal{X}_{f}})$} \label{semirigiXf}

We first recall several definitions of notable classes of Kähler manifolds equipped with torus actions, following \cite{ACGT04}.

\begin{define}\label{d:rigide}
Let \((X, \omega)\) be a connected Kähler manifold with an effective isometric Hamiltonian action of a torus \(\mathbb{T}\) with moment map \(\mu_{\omega}^{\mathbb{T} }\colon X \to \mathfrak{t}^*\). We say the action is \textit{rigid} if for all \(x \in X\), \(R_x^* g\) depends only on \(\mu(x)\), where \(R_x \colon \mathbb{T} \to \mathbb{T} \cdot x \subset X\) is the orbit map and $g$ is the Riemannian metric corresponding to $\omega.$
\end{define}

\begin{define}\label{d:semisimple}
Let \((X, \omega)\) be a connected compact Kähler manifold with an effective isometric Hamiltonian action of a torus \(\mathbb{T}\) with moment map \(\mu_{\omega}^{\mathbb{T}} \colon X \to \mathfrak{t}^*\). Note that for any regular value $\lambda$ of the moment map, we can define a Kähler quotient $(S,\omega_{\lambda})$, where the complex quotient $S$ is independent of $\lambda$ on each connected component of the set of regular values. We say that the torus action is semisimple if for any regular value \(h\) of the moment map \(\mu_{\omega}^{\mathbb{T}}\), the derivative with respect to \(\mu_{\omega}^{\mathbb{T}}\) of the family \(\omega_\lambda\) of Kähler forms on the complex quotient \(S\) is parallel and diagonalizable with respect to \(\omega_h\) at \(\mu_{\omega}^{\mathbb{T}} = h\). 
\end{define}

We show:
\begin{prop}\label{rdss}
The $\tl\times \mathbb{S}^{1}$-action on $(\mathcal{X}_{f},\omega_{\mathcal{X}_{f}})$ is rigid and semisimple.
\end{prop}

\begin{proof}
We know that $\Delta_{R-f,\mathrm{Kir}}$ is exactly the Kirwan polytope of $(\mathcal{X}_{f},\omega_{\mathcal{X}_{f}})$ under the action of $G\times \mathbb{S}^{1}$ and $\Delta_{R-f}$ is the symplectic moment polytope of  $(\mathcal{X}_{f},\omega_{\mathcal{X}_{f}})$ under the action of $\tl\times \mathbb{S}^{1}.$ What's more, there is a natural isomorphism $\Delta_{R-f,\mathrm{Kir}}\rightarrow \Delta_{R-f}.$

We know that $(\mathcal{X}_{f},\omega_{\mathcal{X}_{f}})$ is multiplicity-free in the Hamiltonian sense with respect to the $G\times \mathbb{S}^{1}$-action by \cite[Proposition 9.2.4]{woodward2010moment}, then by \cite[Appendix]{Woodwardmultifree1996} or \cite[Theorem 9.2.1]{woodward2010moment} we know that the points in $\Delta_{R-f,\mathrm{Kir}}$ correspond bijectively to $G\times \mathbb{S}^{1}$-orbits in $\mathcal{X}_{f}$. 

The Riemannian metric $g_{\mathcal{X}_{f}}$ corresponding to $\omega_{\mathcal{X}_{f}}$ is $G\times \mathbb{S}^{1}$-invariant. Then by the identification of $G\times \mathbb{S}^{1}$-orbits with points in $\Delta_{R-f}$ and the fact that the action of $\tl\times \mathbb{S}^{1}$ commutes with that of $G\times \mathbb{S}^{1}$, we know that the $\tl\times \mathbb{S}^{1}$-action is rigid on $(\mathcal{X}_{f},\omega_{\mathcal{X}_{f}}).$

Then we show that $\omega_{\mathcal{X}_{f}}$ is semisimple with respect to the $\mathbb{T}\times \mathbb{S}^{1}$-action. We denote the interior of $\Delta_{R-f}$ as $\mathring{\Delta}_{R-f}.$ First note that $(\mu_{\omega_{\mathcal{X}_{f}}}^{\tl\times \mathbb{S}^{1}})^{-1}(\mathring{\Delta}_{R-f}) = \mathring{\mathcal{X}_{f}}.$ This follows from the identification $\Delta_{R-f,\mathrm{Kir}}\rightarrow \Delta_{R-f},$ \cite[Proposition 5.3.2]{brion1997varietes} and Theorem \ref{kir_alg_id} .

For any point $\lambda$ in $\mathring{\Delta}_{R-f}$, we can also think of it as a point in the relative interior of $\Delta_{R-f,\mathrm{Kir}}$ and we denote its projection onto $\mathfrak{g}^{*}$ as $\bar{\lambda}$. The pair $(\mathcal{X}_{f},\omega_{\mathcal{X}_{f}})$ is a $G\times (\mathbb{T}\times \mathbb{S}^{1})$-Hamiltonian manifold. We can then take the Kähler reduction at $(0,\lambda)$ by adopting the Marsden–Weinstein–Meyer reduction theorem (see for example \cite[Proposition 1.2]{Apo}) to have a $G$-Hamiltonian manifold $(\mu_{\omega_{\mathcal{X}_{f}}}^{\tl\times \mathbb{S}^{1}})^{-1}(\lambda)/(\tl\times \mathbb{S}^{1})$ with a Kähler structure. Note that complex quotient of $\mathring{\mathcal{X}_{f}}$ by $\tl^{\mathbb{C}}\times \mathbb{C}^{*}$ is $\prod_{j=0}^{\ell}\mathbb{P}(\mathbb{C}^{d_{j}})$. So the Kähler reduction is $\prod_{j=0}^{\ell}\mathbb{P}(\mathbb{C}^{d_{j}})$ with a Kähler form whose Kirwan polytope of the moment map with respect to the $\prod_{j=0}^{\ell}U(d_{j})$-action is $\bar{\lambda}.$ The Kähler form is uniquely determined by this Kirwan polytope. We can already conclude semisimplicity here. But let us provide more details.

We know that $(\mu_{\omega_{\mathcal{X}_{f}}}^{\tl\times \mathbb{S}^{1}})^{-1}(\lambda)$ for $\lambda$ in $\mathring{\Delta}_{R-f}$ is a principal $\tl\times \mathbb{S}^{1}$-bundle over $\prod_{j=0}^{\ell} \mathbb{P}(\mathbb{C}^{d_{j}})$. We call it $P_{\tl\times \mathbb{S}^{1}}$. This bundle can be described explicitly in the following way.

We have the Hopf fibration, $$S^{2d_{j}-1}\rightarrow \frac{U(d_{j})}{U(1)\times U(d_{j}-1)}\cong \mathbb{P}(\mathbb{C}^{d_{j}})$$ which provides a $\mathbb{S}^{1}$-principal bundle. We call this principal $\mathbb{S}^{1}$-bundle $P_{j}$ and pull it back to a $\mathbb{S}^{1}$-principal bundle over $\prod_{j=0}^{\ell} \mathbb{P}(\mathbb{C}^{d_{j}})$ which we still call $P_{j}$. Then we have a $(\mathbb{S}^{1})^{\ell+1}\times \mathbb{S}^{1}$-principal bundle $(\oplus_{j=0}^{\ell}P_{j} )\times \mathbb{S}^{1}$. Afterwards, $P_{\tl\times \mathbb{S}^{1}}$ is exactly the quotient of $(\oplus_{j=0}^{\ell}P_{j} )\times \mathbb{S}^{1}$ by the diagonal $\mathbb{S}^{1}$-action on the  $\oplus_{j=0}^{\ell}P_{j}$ part.

Then we have a homomorphism between principal bundles (see \cite[Chapter 1, Section 5]{kobayashi1996foundations}) $(\oplus_{j=0}^{\ell}P_{j} )\times \mathbb{S}^{1}\rightarrow P_{\tl\times \mathbb{S}^{1}}$.

Note that $P_{j}$ is the unitary bundle inside $\mathcal{O}_{\mathbb{P}(\mathbb{C}^{d_{j}})}(-1).$ So there is a natural connection on $P_{j}$ whose curvature is $i\omega_{\mathrm{FS}_{j}}$ (i, the imaginary root is considered as the basis for $\mathrm{Lie}(\mathbb{S}^{1}) = i\mathbb{R}$). Afterwards, $(\oplus_{j=0}^{\ell}P_{j} )\times \mathbb{S}^{1}$ admits a natural connection. Hence, there is an induced connection $\theta_{\tl\times \mathbb{S}^{1}}$ (see \cite[Chapter 2, Section 6]{kobayashi1996foundations}) on $P_{\tl\times \mathbb{S}^{1}}$ whose curvature is
\[\Omega_{\tl\times \mathbb{S}^{1}} = \sum_{j=1}^{\ell }(\omega_{\mathrm{FS}_{j}}-\omega_{\mathrm{FS}_{0}}) \otimes \xi_{j}.\]

Note that $\Omega_{\tl\times \mathbb{S}^{1}}$ on $P_{\tl\times \mathbb{S}^{1}}$ can descend to a form on $\prod_{j=0}^{\ell}\mathbb{P}(\mathbb{C}^{d_{j}})$ which we still denote as $\Omega_{\tl\times \mathbb{S}^{1}}.$

Then note that the Kähler form on the symplectic quotient $(\mu_{\omega_{\mathcal{X}_{f}}}^{\tl\times \mathbb{S}^{1}})^{-1}(\lambda)/(\mathbb{T}\times \mathbb{S}^{1})\cong \prod_{j=0}^{\ell}\mathbb{P}(\mathbb{C}^{d_{j}})$ is exactly $\omega_{\mathrm{FS}_{0}} + \langle \Omega_{\tl\times \mathbb{S}^{1}},\lambda\rangle.$ From this we conclude the semisimplicity.
\end{proof}

\subsection{Generalized Calabi construction of $(\mathcal{X}_{f},\omega_{\mathcal{X}_{f}})$} \label{gencalabiconsxf}

Let us get more descriptions of $(\mathcal{X}_{f},\omega_{\mathcal{X}_{f}})$ following \cite{ACGT04}. For $X = \mathbb{P}\left(\oplus_{j=0}^{\ell} \mathbb{C}^{d_j}\right)$, we have the open part
$$\mathring{X} = \{[\mathbf{z}_0 : \mathbf{z}_1 : \dots : \mathbf{z}_\ell]  \, | \, \mathbf{z}_{j}\neq 0\}.$$ This open part $\mathring{X}$ admits a natural projection to $\prod_{j=0}^{\ell}\mathbb{P}(\mathbb{C}^{d_{j}})$ which realizes $\mathring{X}$ as a $\mathbb{T}^{\mathbb{C}}$-principal bundle over $\prod_{j=0}^{\ell}\mathbb{P}(\mathbb{C}^{d_{j}}).$
Then we consider the open part $\mathring{\mathcal{X}_{f}}\cong \mathring{X}\times \mathbb{C}^{*}\subset \mathcal{X}_{f}$ which is a $\mathbb{T}^{\mathbb{C}}\times \mathbb{C}^{*}$-principal bundle over $\prod_{j=0}^{\ell}\mathbb{P}(\mathbb{C}^{d_{j}}).$ As before, let $P_{\tl\times \mathbb{S}^{1}}$ denote the $\mathbb{T}\times \mathbb{S}^{1}$-principal bundle on $\prod_{j=0}^{\ell}\mathbb{P}(\mathbb{C}^{d_{j}})$. Then we have the identification $\mathring{\mathcal{X}_{f}}\cong P_{\tl\times \mathbb{S}^{1}}\times_{\tl\times \mathbb{S}^{1}}(\mathbb{T}^{\mathbb{C}}\times \mathbb{C}^{*}).$ 

By rigidity and semisimplicity, $\omega_{\mathcal{X}_{f}}$ comes from a generalized Calabi construction (see \cite[Section 2.5]{ACGT04}). More precisely, we have a Kähler toric manifold $(\mathcal{Z}_f,\omega_{\mathcal{Z}_f})$ of complex dimension $\ell+1$ which is a $\mathbb{T}^{\mathbb{C}}\times \mathbb{C}^{*}$-equivariant embedding of $\mathring{\mathcal{Z}}_{f}\cong \mathbb{T}^{\mathbb{C}}\times \mathbb{C}^{*}$, with a moment map $\mu_{\omega_{\mathcal{Z}_f}}^{\mathbb{T}\times \mathbb{S}^{1}}\rightarrow \Delta_{R-f}$ . The momentum-angle coordinate provides $\mathring{\mathcal{Z}}_{f}\cong (\tl\times \mathbb{S}^{1}) \times \mathring{\Delta}_{R-f}.$ We can then write $$\mathring{\mathcal{X}_{f}}\cong P_{\tl\times \mathbb{S}^{1}}\times_{\tl\times \mathbb{S}^{1}}(\mathbb{T}^{\mathbb{C}}\times \mathbb{C}^{*}) \cong P_{\tl\times \mathbb{S}^{1}}\times_{\tl\times \mathbb{S}^{1}} \mathring{\mathcal{Z}}_{f} \cong  P_{\tl\times \mathbb{S}^{1}} \times \mathring{\Delta}_{R-f}.$$ By $\mathring{\mathcal{X}_{f}}\cong P_{\tl\times \mathbb{S}^{1}}\times_{\tl\times \mathbb{S}^{1}} \mathring{\mathcal{Z}}_{f}$, the $\mu_{\omega_{\mathcal{Z}_{f}}}^{\mathbb{T}\times \mathbb{S}^{1}}$ on $\mathring{\mathcal{Z}}_{f}$ descends to a map on $\mathring{\mathcal{X}_{f}}$ which is exactly $\mu_{\omega_{\mathcal{X}_{f}}}^{\tl\times \mathbb{S}^{1}}.$ Moreover, for $\mathring{\mathcal{X}_{f}} \cong P_{\tl\times \mathbb{S}^{1}} \times \mathring{\Delta}_{R-f}$ we have,
\[\omega_{\mathcal{X}_{f}} = \omega_{\mathrm{FS}_{0}}+\langle \mu_{\omega_{\mathcal{Z}_f}}^{\mathbb{T}\times \mathbb{S}^{1}}, \Omega_{\tl\times \mathbb{S}^{1}} \rangle + \langle d\mu_{\omega_{\mathcal{Z}_f}}^{\mathbb{T}\times \mathbb{S}^{1}} \wedge \theta_{\tl\times \mathbb{S}^{1}} \rangle,\]
 where we identify $\mu_{\omega_{\mathcal{Z}_f}}^{\mathbb{T}\times \mathbb{S}^{1}}$ as the coordinate on $ \mathring{\Delta}_{R-f}.$

\section{Compatible test configurations on projective bundles}\label{diffcalDFinv}

\subsection{Compatible K\"ahler metric on projective bundles}{\label{s:comp:bundle}}
We consider
\[
Y := \mathbb{P}\!\left(\oplus_{j=0}^{\ell} E_j\right) \longrightarrow C,
\]
the projectivization of the vector bundle $\oplus_{j=0}^{\ell} E_j$ over a compact smooth curve $C$ with a constant scalar curvature Kähler metric $\omega_{C}$ of integral $2\pi$. We briefly recall  the construction of a compatible K\"ahler metric on $Y$ following \cite{ACGT11}. Each $E_j$ is assumed to be projectively flat. This means that each $E_j$ admits a Hermitian metric $h_{j}$ whose Chern curvature is $-i\mu(E_{j})\text{Id}\otimes \omega_{C}$, where $\mu(E_{j})$ is the slope of $E_{j}$ and $\text{Id}$ is the identity morphism of the bundle.
Let
\begin{equation*}
    \tilde{Y} := \mathbb{P}\left(\oplus_{j=0}^\ell \mathcal{O}_{{\mathbb{P}(E_j)}}(-1)\right) \longrightarrow C
\end{equation*}
be the relative birational resolution of $Y$, obtained fiberwise by resolving the rational projections onto the projectivized summands, see Example \ref{e:proj:bund}. We denote by
\(
\tilde{X} := \mathbb{P}\left(\oplus_{j=0}^\ell \mathcal{O}_{\mathbb{P}(\mathbb{C}^{d_j})}(-1)\right)
\)
its fiber. Consider a degenerate K\"ahler metric $\omega_{\tilde{X}}$ on the fiber $\tilde{X}$, induced by a K\"ahler metric $\omega_{\mathbb{P}^\ell}$ on $\mathbb{P}^\ell$ as explained in Section \ref{s:set-up}. Then $\omega_{\tilde{X}}$ descends to a smooth form on $\tilde{Y}$ (still denoted by $\omega_{\tilde{X}}$), and we define a degenerate smooth form on $\tilde{Y}$ by
\[
\omega_{\tilde{Y}} := \omega_{\tilde{X}} + \big(\langle \mu, p_C \rangle + c-\mu(E_{0})\big)\,\omega_C, 
\]
where $c \in \mathbb{R}$ is any constant such that $c>\mu(E_{j})$ for any $0\leq j\leq \ell$, $\mu$ denotes the moment map of $\omega_{\mathbb{P}^\ell}$ and  $p_C=\sum_{j=1}^\ell (\mu(E_0)-\mu(E_j))\otimes \xi_j$.
Here $\mu(E_j)$ is the slope of $E_j$.

It follows from \cite[Proposition~2]{ACGT06} that $\omega_{\tilde{Y}}$ is the pull-back, via the blow-down map $\mathrm{Bd} : \tilde{Y} \rightarrow Y$, of a unique K\"ahler metric $\omega_Y$ on $Y$.

Following \cite{ACGT06, ACGT04}, we make the following definition.

\begin{define}\label{d:comp:metric:2}
A K\"ahler metric $\omega_Y$ on $Y$ constructed as above from a K\"ahler metric $\omega_{\mathbb{P}^\ell}$ on $\mathbb{P}^\ell$ is called a compatible K\"ahler metric.
\end{define}

Observe that the compatible K\"ahler metrics defined in Definition~\ref{d:comp:metric} are a special case of those in Definition~\ref{d:comp:metric:2} when the base $C$ reduces to a point.

\begin{remark}It follows from \cite{ACGT06, ACGT04} that the compatible K\"ahler metrics $\omega_Y$ on $Y$ are precisely those for which the $\T$-action on $(Y,\omega_Y)$ is rigid and semisimple, in the sense of Definitions \ref{d:rigide} and \ref{d:semisimple}.
\end{remark}

\subsection{Compatible test configurations for projective bundles}{\label{s:bundle:tc}}
Let $h_j$ be the Hermitian-Einstein metric on $E_j$ for $j = 0,\cdots, \ell.$ 
Let $G = U(d_0) \times \cdots \times U(d_\ell)$. Let $P$ be the principal $G$-bundle over $C$ corresponding to $(E,\oplus_{j=0}^{\ell} h_j)$ with the curvature in the form:
\[
\Omega = - i 
\begin{pmatrix}
\mu(E_{0}) \text{Id} & \cdots & 0 \\
\vdots & \ddots & \vdots \\
0 & \cdots & \mu(E_{l}) \text{Id}
\end{pmatrix} 
\otimes \omega_C,
\]
where $\mu(E_{j})$ is the slope of $E_{j}.$

From now on, we assume that $\Delta_{R-f}$ is a Delzant polytope. In Subsection \ref{semirigiXf} we construct the Kähler manifold $(\mathcal{X}_{f},\omega_{\mathcal{X}_{f}})$ with the corresponding moment map $\mu_{\omega_{\mathcal{X}_{f}}}^{G}:\mathcal{X}_{f}\rightarrow \mathfrak{g}^{*}.$ Then by using the bundle construction described in Section \ref{bundleconstruction}, we have a Kähler manifold

\begin{equation*}
    \mathcal{Y}_{f} :=P\times_{G}\mathcal{X}_{f}
\end{equation*}
with the Kähler form $$\omega_{\mathcal{Y}_{f}}:=c\omega_{C} + \langle \mu_{\omega_{\mathcal{X}_{f}}}^{G}, \Omega \rangle + \omega_{\mathcal{X}_{f}}$$ for $c> \max\{\mu(E_{j})\}$. See Subsection \ref{Form_constructions} for details about $\omega_{\mathcal{X}_{f}}$. 

Note that we have a $\mathbb{T}\times \mathbb{S}^{1}$-action on $\mathcal{Y}_{f}$ induced by the $\mathbb{T}\times \mathbb{S}^{1}$-action on $\mathcal{X}_{f}$. It is easy to see that $\omega_{\mathcal{Y}_{f}}$ is $\mathbb{T}\times \mathbb{S}^{1}$-invariant. Although $\mu_{\omega_{\mathcal{X}_{f}}}^{G}$ cannot descend from $P\times \mathcal{X}_{f}$ to $\mathcal{Y}_{f}$, the moment map $\mu_{\omega_{\mathcal{X}_{f}}}^{\mathbb{T}\times \mathbb{S}^{1}}$ can since this map is $G$-invariant (it is actually $G\times \mathbb{S}^{1}$-invariant). We denote the corresponding map on $\mathcal{Y}_{f}$ as $\mu_{\omega_{\mathcal{Y}_{f}}}^{\mathbb{T}\times \mathbb{S}^{1}}.$  We have the following proposition.

\begin{prop}\label{Yfssrgd}
The map $\mu_{\omega_{\mathcal{Y}_{f}}}^{\mathbb{T}\times \mathbb{S}^{1}}$ is a moment map of $(\mathcal{Y}_{f},\omega_{\mathcal{Y}_{f}})$ for the action of $\mathbb{T}\times \mathbb{S}^{1}$. Moreover, the $\mathbb{T}\times \mathbb{S}^{1}$-action on $(\mathcal{Y}_{f},\omega_{\mathcal{Y}_{f}})$ is semisimple and rigid.
\end{prop}

\begin{proof}
It is direct to check the first statement. The rigidity follows directly from Proposition \ref{rdss}. Let us focus on the semisimplicity. For this, we need to compute the Kähler reduction for each point $\lambda$ in $\mathring{\Delta}_{R-f}.$

Note that $\mathcal{Y}_{f}$ contains the following open part:
\[\mathring{\mathcal{Y}}_{f}:= P\times_{G} \mathring{\mathcal{X}}_{f}\]
which is a $\T^{\mathbb{C}}\times \mathbb{C}^{*}$-principal bundle over $S:= P\times_{G} \prod_{j=0}^{\ell}\mathbb{P}(\mathbb{C}^{d_{j}}) = \mathbb{P}(E_{0})\times_{C}\mathbb{P}(E_{1})\times_{C}\cdots \times_{C}\mathbb{P}(E_{\ell}).$ Recall that $\mathring{\mathcal{X}}_{f}$ is introduced in Subsection \ref{gencalabiconsxf}. So as a complex manifold, $(\mu_{\omega_{\mathcal{Y}_{f}}}^{\mathbb{T}\times \mathbb{S}^{1}})^{-1}(\lambda)/(\tl\times \mathbb{S}^{1})$ is $S$.

Note that the $G-$horizontal direction  of $T\mathcal{Y}_{f}$ is orthogonal to the vertical distribution of $T\mathcal{Y}_{f}$ with respect to the Riemannian metric associated with $\omega_{\mathcal{Y}_{f}}.$ This implies that the induced Kähler form, $\omega_\lambda$ on $S$ decomposes into two parts: the first is 
\begin{equation}
\begin{aligned}
&c\omega_{C} - \mu(E_{0})\omega_{C}+ \langle \sum_{j=1}^{\ell}(\mu(E_{0})-\mu(E_{j}))\xi_{j},\lambda\rangle \omega_{C} \\
=&c\omega_{C} - \mu(E_{0})\omega_{C} - \langle \sum_{j=0}^{\ell} \mu(E_{j})p_{j},\lambda\rangle\omega_{C} 
\end{aligned}
\end{equation}
which is induced by $c\omega_{C} + \langle \mu_{\omega_{\mathcal{X}_{f}}}^{G}, \Omega \rangle$ and is defined on the $G$-horizontal distribution of $S$. Recall that $p_{j} = \xi_{j}$ for $1\leq j\leq \ell$ and $p_{0} = -\sum_{j=1}^{\ell}\xi_{j}.$ The second is the form induced on the vertical distribution of $TS$ by
$(\mu_{\omega_{\mathcal{X}_{f}}}^{\tl \times \mathbb{S}^{1}})^{-1}(\lambda)/(\tl \times \mathbb{S}^{1})$
from the proof of Proposition~\ref{rdss}. It follows that the induced Kähler form $\omega_\lambda$ on $S$ is,
\begin{equation}\label{omega_lambda}
\begin{aligned}
&c\omega_{C} - \mu(E_{0}) \omega_{C} + \omega_{\mathrm{FS}_{0}} + \langle \sum_{j=1}^{\ell}(\omega_{\mathrm{FS}_{j}}-\omega_{\mathrm{FS}_{0}} + (\mu(E_{0})-\mu(E_{j}))\omega_{C})\otimes\xi_{j},\lambda\rangle\\
=& c\omega_{C} - \mu(E_{0}) \omega_{C} + \omega_{\mathrm{FS}_{0}} - \langle \sum_{j=0}^{\ell} \mu(E_{j})p_{j},\lambda\rangle\omega_{C} +  \sum_{j=0}^{\ell}\langle p_{j}, \lambda  \rangle \omega_{\mathrm{FS}_{j}}
\end{aligned}
\end{equation}

We immediately observe that the family of Kähler forms \( \omega_{\lambda} \) on \( S \), parametrized by \( \lambda \in \mathring{\Delta}_{R-f} \), is simultaneously diagonalizable. In particular, this implies that the derivative of \( \omega_{\lambda} \) with respect to \( \lambda \) is diagonalizable at \( \lambda = h \) with respect to $\omega_{h}$, for any \( h \in \mathring{\Delta}_{R-f} \) .

The Kähler form $\omega_{\lambda}$ is locally a product. Let us explain this more explicitly. We have the $G$-principal bundle $P$, then we have the $\tilde{G} = \prod_{i=0}^{\ell}PU(d_{i})$-principal bundle $\tilde{P}$ over $C.$ We have a natural principal bundle morphism with them. This provides a flat connection on $\tilde{P}$ since the connection of $P$ is projectively-flat.  By the flatness of the connection on $\tilde{P},$ we know that locally on a simply connected open set $D,$ $\tilde{P}|_{D}\cong D\times \tilde{G}$ and the horizontal distribution associated with the connection on $\tilde{P}$ is exactly $TD.$ 

We have $S=P\times_{G} \prod_{j=0}^{\ell}\mathbb{P}(\mathbb{C}^{d_{j}})= \tilde{P}\times_{\tilde{G}}\prod_{j=0}^{\ell}\mathbb{P}(\mathbb{C}^{d_{j}}),$ so locally $S|_{D}\cong D\times \prod_{j=0}^{\ell}\mathbb{P}(\mathbb{C}^{d_{j}}).$ With this identification, the Kähler form $\omega_{\lambda}$ on $S|_{D}$ is exactly a product of some multiple of $\omega_{C}$ on $D$ and some multiple of $\omega_{\mathrm{FS}_{j}}$ on each $\mathbb{P}(\mathbb{C}^{d_{j}})$. This follows from the expression \eqref{omega_lambda} and the fact that the $G$-horizontal distribution and the $\tilde{G}$-horizontal distribution of $S$ are the same since the connection on $\tilde{P}$ is induced by the connection on $P.$ This implies that the derivative of \( \omega_{\lambda} \) with respect to \( \lambda \) is parallel with respect to the Levi–Civita connection associated with $\omega_{h}$ at \( \lambda = h \), for any \( h \in \mathring{\Delta}_{R-f} \).

We can already conclude semisimplicity at this stage. But let us see the Kähler form on $S$ in a clearer manner. 

We have $(\mu_{\omega_{\mathcal{Y}_{f}}}^{\mathbb{T}\times \mathbb{S}^{1}})^{-1}(\lambda) \cong P\times_{G}P_{\tl\times \mathbb{S}^{1}},$ where $P_{\tl\times \mathbb{S}^{1}}$ is the $\tl\times \mathbb{S}^{1}$-principal bundle over $\prod_{j=0}^{\ell}\mathbb{P}(\mathbb{C}^{d_{j}})$ introduced in the proof of Proposition \ref{rdss}. In the proof of Proposition \ref{rdss}, we also introduced the $(\mathbb{S}^{1})^{\ell+1}\times \mathbb{S}^{1}$-principal bundle $(\oplus_{j=0}^{\ell}P_{j} )\times \mathbb{S}^{1}$ as well as the principal bundle morphism $(\oplus_{j=0}^{\ell}P_{j} )\times \mathbb{S}^{1}\rightarrow P_{\tl\times \mathbb{S}^{1}}.$  With the bundle construction, we also have the $(\mathbb{S}^{1})^{\ell+1}\times \mathbb{S}^{1}$-principal bundle $P\times_{G} ((\oplus_{j=0}^{\ell}P_{j} )\times \mathbb{S}^{1})$ and a bundle morphism $P\times_{G} ((\oplus_{j=0}^{\ell}P_{j} )\times \mathbb{S}^{1})\to P\times_{G}P_{\tl\times \mathbb{S}^{1}}.$ 

Note that $P\times_{G}P_{j}$ is just the unitary bundle inside $\mathcal{O}_{\mathbb{P}(E_{j})}(-1).$ There is a natural connection on it induced by the Chern connection of $(\mathcal{O}_{\mathbb{P}(E_{j})}(-1),h_{j})$, the corresponding Chern curvature is $i(\omega_{FS_{j}}-\mu(E_{j})\omega_{C}).$ Then we have a natural connection on the principal bundle $P\times_{G} ((\oplus_{j=0}^{\ell}P_{j} )\times \mathbb{S}^{1})$, which we still call the Chern connection, whose curvature is 
\begin{equation}
\sum_{j=0}^{\ell}(\omega_{\mathrm{FS}_{j}} - \mu(E_{j})\omega_{C})\otimes\xi_{j}
\end{equation}

The principal bundle morphism $P\times_{G} ((\oplus_{j=0}^{\ell}P_{j} )\times \mathbb{S}^{1})\to P\times_{G}P_{\tl\times \mathbb{S}^{1}}$ provides an induced connection $\hat{\theta}_{\tl\times \mathbb{S}^{1}}$ (see \cite[Chapter 2, Section 6]{kobayashi1996foundations}) on $P\times_{G}P_{\tl\times \mathbb{S}^{1}}$. The curvature $d\hat{\theta}_{\tl\times \mathbb{S}^{1}}$ is
\begin{equation}
\begin{aligned}
d\hat{\theta}_{\tl\times \mathbb{S}^{1}} =& \sum_{j=1}^{\ell}(\omega_{\mathrm{FS}_{j}}-\omega_{\mathrm{FS}_{0}} + (\mu(E_{0})-\mu(E_{j}))\omega_{C})\otimes\xi_{j}\\
=&\sum_{j=0}^{\ell}(\omega_{\mathrm{FS}_{j}} - \mu(E_{j})\omega_{C})\otimes p_{j}
\end{aligned}
\end{equation}
The curvature $d\hat{\theta}_{\tl\times \mathbb{S}^{1}}$ can descend to $S$. Then we see that the  Kähler  form $\omega_\lambda$ on $(\mu_{\omega_{\mathcal{Y}_{f}}}^{\mathbb{T}\times \mathbb{S}^{1}})^{-1}(\lambda)/(\tl\times \mathbb{S}^{1})\cong S$ can be written as,
\[ \omega_\lambda = c\omega_{C} - \mu(E_{0}) \omega_{C} + \omega_{\mathrm{FS}_{0}} + \langle d\hat{\theta}_{\tl\times \mathbb{S}^{1}},\lambda \rangle.\]

This completes the proof of the semisimplicity.
\end{proof}

Like in Section \ref{gencalabiconsxf}, we can say more about $\omega_{\mathcal{Y}_{f}}$ following \cite{ACGT04}. We have the Kähler toric manifold $(\mathcal{Z}_{f},\omega_{\mathcal{Z}_{f}})$ from Subsection \ref{gencalabiconsxf}. The momentum-angle coordinate provides $\mathring{\mathcal{Z}}_{f}\cong (\tl\times \mathbb{S}^{1}) \times \mathring{\Delta}_{R-f}.$ Denote $P\times_{G}P_{\tl\times \mathbb{S}^{1}}$ as $\mathbf{P}_{\tl\times \mathbb{S}^{1}}$. We can write $$\mathring{\mathcal{Y}}_{f}\cong \mathbf{P}_{\tl\times \mathbb{S}^{1}}\times_{\tl\times \mathbb{S}^{1}}(\mathbb{T}^{\mathbb{C}}\times \mathbb{C}^{*}) \cong (\mathbf{P}_{\tl\times \mathbb{S}^{1}})\times_{\tl\times \mathbb{S}^{1}} \mathring{\mathcal{Z}}_{f} \cong  \mathbf{P}_{\tl\times \mathbb{S}^{1}} \times \mathring{\Delta}_{R-f}.$$ By $\mathring{\mathcal{Y}}_{f}\cong (\mathbf{P}_{\tl\times \mathbb{S}^{1}})\times_{\tl\times \mathbb{S}^{1}} \mathring{\mathcal{Z}}_{f}$, the $\mu_{\omega_{Z}}^{\mathbb{T}\times \mathbb{S}^{1}}$ on $\mathring{\mathcal{Z}}_{f}$ descends to a map on $\mathring{\mathcal{Y}}_{f}$ which turns out to be exactly $\mu_{\omega_{\mathcal{Y}_{f}}}^{\tl\times \mathbb{S}^{1}}.$ Moreover, for $\mathring{\mathcal{Y}}_{f} \cong \mathbf{P}_{\tl\times \mathbb{S}^{1}} \times \mathring{\Delta}_{R-f}$ we have,
\[\omega_{\mathcal{Y}_{f}} = c\omega_{C} - \mu(E_{0}) \omega_{C}+\omega_{\mathrm{FS}_{0}}+\langle \mu_{\omega_{Z}}^{\mathbb{T}\times \mathbb{S}^{1}}, d\hat{\theta}_{\tl\times \mathbb{S}^{1}} \rangle + \langle d\mu_{\omega_{Z}}^{\mathbb{T}\times \mathbb{S}^{1}}\wedge\hat{\theta}_{\tl\times \mathbb{S}^{1}} \rangle,\]
where we identify $\mu_{\omega_{\mathcal{Z}_{f}}}^{\mathbb{T}\times \mathbb{S}^{1}}$ as the coordinate on $ \mathring{\Delta}_{R-f}.$

\begin{prop}
The Kähler form $\omega_{\mathcal{Y}_{f}}$ represents the Kähler class $2\pi[\mathcal{O}_{\mathcal{Y}_{f}}(1)+cF]$.
\end{prop}

\begin{proof}

Let $V' = H^{0}(\mathcal{X}_{f},k\mathcal{L}_{f})^{*}$.
We have the embedding $\mathcal{Y}_{f}\rightarrow \mathbb{P}(V')\cong P^{\mathbb{C}}\times_{G^{\mathbb{C}}}\mathbb{P}(H^{0}(\mathcal{X}_{f},k\mathcal{L}_{f})^{*})\cong P\times_{G}\mathbb{P}(H^{0}(\mathcal{X}_{f},k\mathcal{L}_{f})^{*})$. Then we can construct a Kähler form $\omega_{\mathbb{P}(V')} = \langle \mu_{\mathbb{P}(H^{0}(\mathcal{X}_{f},k\mathcal{L}_{f})^{*}),G},\Omega\rangle + \omega_{\mathrm{FS}} + kc \omega_{C}$ on $\mathbb{P}(V')$ following Section \ref{bundleconstruction}. As in \cite[Page 12]{AK}, $\omega_{\mathbb{P}(V')}$ represents $2
\pi[\mathcal{O}_{\mathbb{P}(V')}(1)+kcF]$. Then note that the pullback of $\omega_{\mathbb{P}(V')}$ to $\mathcal{Y}_{f}$ is exactly $k\omega_{\mathcal{Y}_{f}}$.
\end{proof}

We denote the Kähler class of $\omega_{\mathcal{Y}_{f}}$ as $\mathcal{A}_{\mathcal{Y}_f}.$ 

\begin{define}
A test configuration $(\mathcal{Y}_f, \mathcal{A}_{\mathcal{Y}_f})$ for $(Y,[\omega_Y],\T)$ constructed as above is called a compatible test configuration.
\end{define}

\subsection{Proof of Theorem \ref{t:B}}

We recall that the weights $p$, $\bar{p}$, $\hat{w}$, and $\bar{w}$ are defined in Example \ref{e:proj:bund}.

\begin{theorem}{\label{t:stab:y:x}}
 Consider the projectivization of a holomorphic bundle over a curve $Y=\mathbb{P}(E)$ with fixed K\"ahler class $[\omega_Y]$. Let $(X,[\omega_X])$ be its fiber. Then, $(Y,[\omega_Y])$ is $\T$-relatively uniformly K-stable for compatible test configurations if and only if $(X,[\omega_X])$ is $(\bar{p},\hat{w})$-uniformly K-stable for compatible test configurations.
\end{theorem}

\begin{proof}
Let $n=\dim_{\mathbb C}X$, and set
\[
V_C:=\int_C\omega_C,\qquad
V_X:=\int_X\omega_X^{[n]},\qquad
V_Y:=\int_Y\omega_Y^{[n+1]}.
\]

We fix a compatible K\"ahler metric
$\omega_{\Y_f}\in\A_{\Y_f}$ induced by a compatible K\"ahler metric
$\omega_{\X_f}$ on $\X_f$.
By Proposition \ref{Yfssrgd}, the $\T$-action on
$(\Y_f,\omega_{\Y_f})$ is rigid and semisimple.
Hence, by \cite[(7) \& (8)]{ACGT11}
(see also Lemma \ref{l:scal}),
\[
\mathrm{Scal}(\omega_{\Y_f})
=
\mathrm{Scal}_{\bar p}(\omega_{\X_f})
+\frac{\mathrm{Scal}(\omega_C)}{\bar p},
\]
and
\begin{equation}{\label{vol:bundle}}
\omega_{\Y_f}^{[n+2]}
=
\bar p(\mu_{\omega_{\X_f}})
\omega_{\X_f}^{[n+1]}\wedge\omega_C.
\end{equation}
We immediately deduce that
\begin{equation}{\label{df:1}}
\mathbf{DF}(\Y_f,\A_{\Y_f})
=
V_C\mathbf{DF}_{\bar p,\hat w}(\X_f,\A_{\X_f}),
\end{equation}
where $\hat w$ is the smooth second weight on $X$,
defined in Example \ref{e:proj:bund}.
Recall that the corresponding reduced weight on
$\Delta$ is $\bar w$.

We now compare the non-Archimedean $J$-functionals.
By \eqref{vol:bundle} and \cite[(8)]{ACGT11},
\begin{equation}{\label{j:3}}
\int_{\Y_f}\omega_{\Y_f}^{[n+2]}
=
V_C\int_{\X_f}
\bar p(\mu_{\omega_{\X_f}})
\omega_{\X_f}^{[n+1]},
\end{equation}
and
\begin{equation}{\label{j:2}}
V_Y
=
V_C\int_X
\bar p(\mu_{\omega_X})\omega_X^{[n]}.
\end{equation}

Since $\X_f$ is a dominating $G^{\mathbb C}$-equivariant
test configuration, its dominating morphism
$\Pi:\X_f\to X\times\P^1$ induces a dominating
morphism $\Pi:\Y_f\to Y\times\P^1$.
Using the compatible construction and
$\omega_C^2=0$, we obtain
\begin{equation}{\label{j:1}}
\begin{split}
&\int_{\Y_f}
\Pi^*(\omega_Y^{[n+1]})\wedge\omega_{\Y_f}
 =
V_C\int_{\X_f}
\Pi^*\big(
\bar p(\mu_{\omega_X})\omega_X^{[n]}
\big)\wedge\omega_{\X_f}.
\end{split}
\end{equation}

Combining \eqref{j:3}, \eqref{j:2}, and \eqref{j:1},
and proceeding as in Lemma \ref{l:J-fonc}, we obtain
\begin{equation}{\label{ineq:j}}
\begin{split}
&\frac{1}{(2\pi)^{n+1}}
\mathbf{J}^{\mathrm{NA}}(\Y_f,\A_{\Y_f}) =
\frac{V_C\mathrm{Vol}(\Delta_B)}{V_Y}
\int_\Delta
(f-\inf_\Delta f)p\bar p\,dx.
\end{split}
\end{equation}

On the other hand, by Lemma
\ref{l:J-fonc:compa:toric},
\[
\frac{1}{(2\pi)^{n+1}}
\mathbf{J}^{\mathrm{NA}}(\X_f,\A_{\X_f})
=
\frac{\mathrm{Vol}(\Delta_B)}{V_X}
\int_\Delta
(f-\inf_\Delta f)p\,dx.
\]

Since $\bar p$ is strictly positive on $\Delta$,
we have
\[
\begin{split}\big(\inf_\Delta\bar p\big)
\int_\Delta(f-\inf_\Delta f)p\,dx
&\leq
\int_\Delta(f-\inf_\Delta f)p\bar p\,dx
\\
&\leq
\big(\sup_\Delta\bar p\big)
\int_\Delta(f-\inf_\Delta f)p\,dx.
\end{split}
\]

Consequently,
\[
\begin{split}
\frac{V_CV_X}{V_Y}
\big(\inf_\Delta\bar p\big)
\mathbf{J}^{\mathrm{NA}}(\X_f,\A_{\X_f})
&\leq
\mathbf{J}^{\mathrm{NA}}(\Y_f,\A_{\Y_f})
\\
&\leq
\frac{V_CV_X}{V_Y}
\big(\sup_\Delta\bar p\big)
\mathbf{J}^{\mathrm{NA}}(\X_f,\A_{\X_f}).
\end{split}
\]

These inequalities also hold for the reduced
non-Archimedean $J$-functionals.
Indeed, twisting a compatible test configuration
by $\xi\in\mathfrak t$ amounts to replacing $f$
by $f+\xi$, viewed as an affine function on
$\Delta$. Since the preceding estimates hold
uniformly for every such twist, taking the infimum
over $\xi\in\mathfrak t$ yields
\begin{equation}{\label{ineq:j:reduced}}
\begin{split}
\frac{V_CV_X}{V_Y}
\big(\inf_\Delta\bar p\big)
\mathbf{J}^{\mathrm{NA}}_{\T^{\mathbb C}}
(\X_f,\A_{\X_f})
&\leq
\mathbf{J}^{\mathrm{NA}}_{\T^{\mathbb C}}
(\Y_f,\A_{\Y_f})
\\
&\leq
\frac{V_CV_X}{V_Y}
\big(\sup_\Delta\bar p\big)
\mathbf{J}^{\mathrm{NA}}_{\T^{\mathbb C}}
(\X_f,\A_{\X_f}).
\end{split}
\end{equation}

Suppose first that $(Y,[\omega_Y])$ is
$\T$-relatively uniformly $K$-stable with respect
to compatible test configurations.
Then there exists $\lambda_Y>0$ such that
\[
\mathbf{DF}(\Y_f,\A_{\Y_f})
\geq
\lambda_Y
\mathbf{J}^{\mathrm{NA}}_{\T^{\mathbb C}}
(\Y_f,\A_{\Y_f}).
\]
By \eqref{df:1} and \eqref{ineq:j:reduced},
\[
\begin{split}
\mathbf{DF}_{\bar p,\hat w}
(\X_f,\A_{\X_f})
&\geq
\frac{\lambda_Y}{V_C}
\mathbf{J}^{\mathrm{NA}}_{\T^{\mathbb C}}
(\Y_f,\A_{\Y_f})
\\
&\geq
\lambda_Y
\frac{V_X}{V_Y}
\big(\inf_\Delta\bar p\big)
\mathbf{J}^{\mathrm{NA}}_{\T^{\mathbb C}}
(\X_f,\A_{\X_f}).
\end{split}
\]
This proves the weighted uniform $K$-stability
of $(X,[\omega_X])$ with respect to compatible
test configurations.

Conversely, suppose that $(X,[\omega_X])$ is
$(\bar p,\hat w)$-uniformly $K$-stable with
respect to compatible test configurations.
Then there exists $\lambda_X>0$ such that
\[
\mathbf{DF}_{\bar p,\hat w}
(\X_f,\A_{\X_f})
\geq
\lambda_X
\mathbf{J}^{\mathrm{NA}}_{\T^{\mathbb C}}
(\X_f,\A_{\X_f}).
\]
Using \eqref{df:1} and the upper bound in
\eqref{ineq:j:reduced}, we obtain
\[
\begin{split}
\mathbf{DF}(\Y_f,\A_{\Y_f})
&=
V_C\mathbf{DF}_{\bar p,\hat w}
(\X_f,\A_{\X_f})
\\
&\geq
V_C\lambda_X
\mathbf{J}^{\mathrm{NA}}_{\T^{\mathbb C}}
(\X_f,\A_{\X_f})
\\
&\geq
\lambda_X
\frac{V_Y}{V_X\sup_\Delta\bar p}
\mathbf{J}^{\mathrm{NA}}_{\T^{\mathbb C}}
(\Y_f,\A_{\Y_f}).
\end{split}
\]
This proves the converse implication and
concludes the proof.
\end{proof}

We conclude with the proof of Theorem \ref{t:B}.

\begin{proof}
This follows from Theorem \ref{t:stab:y:x} and Proposition \ref{l:kstab:stab:delta} by choosing $v=\bar{p}$ and $w=\hat{w}$, so that $pv=p\bar{p}$ and the corresponding reduced weight is $\tilde{w}=\bar{w}$.

\end{proof}

\begin{remark}
The formulas for the Donaldson--Futaki invariant and for the non--archimedean $J$-functional are essential to the proof of Theorem~\ref{t:stab:y:x}. A purely algebraic approach to deriving these formulas can be found in \cite[Section~1.2]{yin2025explicit}.
\end{remark}

\appendix

\section{Weighted toric K-stability}{\label{a:toric}}

In this appendix, we recall some results that appear to be well known to experts, but for which we have not been able to find explicit proofs in the literature.

\subsection{Donaldson's construction of toric test configurations}{\label{COnstruction:Donaldson}}
Consider a toric K\"ahler manifold $(Z,\omega_Z,\T)$ with Delzant polytope $(\Delta,\mathbf{L})$. Let $f=\max(f_1,\ldots,f_p)$ be a convex PL function on $\Delta$, where $f_1,\ldots,f_p$ are a minimal collection of affine-linear functions with rational coefficients. After multiplying $f$ by a suitable positive integer, we may assume, when necessary, that these functions have integral coefficients. We choose $R>0$ such that $R-f$ is strictly positive on $\Delta$. We consider the rational polytope
\[
\Delta_{R-f}:=\{(x,y)\in\Delta\times\mathbb{R}\mid 0\leq y\leq R-f(x)\}
\]
with defining inequalities
\[
\mathbf{L}_{R-f}:=\mathbf{L}\cup\{y\geq0,\quad R-y-f_i\geq0,\ i=1,\ldots,p\}.
\]
The normal fan of $\Delta_{R-f}$ determines a normal projective toric variety $\mathcal{Z}_f$ equipped with a toric K\"ahler class $\mathcal{A}_{\Z_f}$ and a $\T\times\S^1$-action. The induced toric morphism to $\P^1$ is a test configuration for $(Z,[\omega_Z],\T)$. 

\begin{define}[\cite{DR17,SKD}]{\label{d:torictc}}
The normal toric variety $(\mathcal{Z}_f,\mathcal{A}_{\Z_f},\T\times\S^1)$ constructed above is called a toric test configuration.
\end{define}

Let $\RPL$ be the set of convex piecewise-linear functions with rational coefficients on $\Delta$. By the above construction, $\RPL$ corresponds to normal toric test configurations of $(Z,[\omega_Z],\T)$. We also define $\DPL$ as the set of convex piecewise-linear functions $f$ such that the polytope $\Delta_{R-f}$ associated with Donaldson's construction is Delzant. This set corresponds to smooth toric test configurations.

Finally, denote by $\Sigma_{R-f}$ the normal fan of $\Delta_{R-f}$ and by $\Sigma_{\mathrm{prod}}$ that of $\Delta\times[0,\sup(R-f)]$. We consider
\[
\mathcal{DPL}_{\mathrm{dom}}(\Delta)
:=\bigl\{f\in\DPL\mid\Sigma_{R-f}\text{ refines }\Sigma_{\mathrm{prod}}\bigr\}.
\]
Any element of $\mathcal{DPL}_{\mathrm{dom}}(\Delta)$ defines a smooth dominating toric test configuration, as explained in Appendix \ref{a1:toric}. In particular, such a function has a constant affine piece, since the horizontal upper ray of the product fan is preserved by refinement.

\subsection{Twist of toric test configurations}

Let $(\Z_f,\A_{\Z_f})$ be a toric test configuration with moment polytope $(\Delta_{R-f},\mathbf{L}_{R-f})$ as in Section \ref{COnstruction:Donaldson}. A direct computation shows that the polytope of the $\xi$-twist $((\mathcal{Z}_f)_\xi,(\mathcal{A}_{\Z_f})_\xi,\T\times\S^1_\xi)$ of $(\Z_f,\A_{\Z_f},\T\times\S^1)$ is $\Delta_{R-f-\xi}$. Hence, the $\xi$-twist of $(\Z_f,\A_{\Z_f})$ is $(\Z_{f+\xi},\A_{\Z_{f+\xi}})$.

\subsection{Weighted K-stability of toric manifolds}

A detailed proof of the following result in the polarized case can be found in \cite[Proposition 4.5.3]{NS}. Since no proof appears to be available to the authors for the transcendental case, we provide a sketch below. Our argument will also allow us to obtain a weighted version in Lemma \ref{l:J-fonc:compa:toric}.

\begin{lemma}{\label{l:J-fonc}}
Let $(\Z_f,\A_{\Z_f},\T\times\S^1)$ be a dominating toric test configuration associated with a function $f\in\mathcal{DPL}_{\mathrm{dom}}(\Delta)$. Then, the reduced $\mathbf{J}_{\T^{\mathbb{C}}}^{\mathrm{NA}}$-functional evaluated on $(\Z_f,\A_{\Z_f},\T\times\S^1)$ is given by
\begin{equation*}
\frac{1}{(2\pi)^{n+1}}\J_{\mathbb{T}^{\mathbb{C}}}(\Z_f,\A_{\Z_f})
=\frac{1}{\mathrm{Vol([\omega]})}\inf_{\xi\in\mathrm{Aff}(\Delta)}
\int_\Delta\bigl(f+\xi-\inf_\Delta(f+\xi)\bigr)\,dx.
\end{equation*}
\end{lemma}

\begin{proof}
Since $f\in\mathcal{DPL}_{\mathrm{dom}}(\Delta)$, there is a dominating toric morphism $\Pi:\mathcal{Z}_f\rightarrow Z\times\P^1$. It follows from \cite[Proposition 3.10]{Sjo} that we can write
\[
\mathcal{A}_{\Z_f}=a_0E_0+\sum_{j=1}^q b_jF_j+\sum_{i=1}^p a_iE_i,
\]
where the strict transform $E_0$ corresponds to the horizontal upper facet $F_{E_0}:=\Delta_{R-f}\cap\{y=\sup(R-f)\}$, the $F_j$ are toric divisors associated with the facets $L_j=0$, and the $E_i$, $i\geq1$, are irreducible components of the exceptional divisor of $\Pi$. With the chosen normalization, the toric divisor--polytope correspondence \cite[Proposition 4.2.14]{CLS11} gives $a_0=\sup(R-f)$. Inserting this expression into \eqref{non:arc:j} and pushing forward by the moment map yields the formula for the unreduced $\mathbf J^{\mathrm{NA}}$-functional. The reduced formula follows by taking the infimum over affine twists, as in \cite[Proposition 4.5.3]{NS}.
\end{proof}

\begin{lemma}{\label{l:toric:fut}}\textnormal{(\cite[Proposition 4]{Lah19})}.
Let $(\Z_f,\A_{\Z_f},\T\times\S^1)$ be a smooth toric test configuration associated with a function $f\in\mathcal{DPL}(\Delta)$. Then, the weighted Donaldson--Futaki invariant of $(\Z_f,\A_{\Z_f},\T\times\S^1)$ is given by
\begin{equation}{\label{obj2}}
\mathbf{DF}_{v,w}(\Z_f,\A_{\Z_f})=(2\pi)^{n+1}\mathcal{F}^{\Delta}_{v,w}(f).
\end{equation}
\end{lemma}

By Lemmas \ref{l:J-fonc} and \ref{l:toric:fut}, and the density result proved in Lemma \ref{l:smooth:test}, we deduce
\begin{lemma}
Let $(Z,\omega,\T)$ be a toric K\"ahler manifold with Delzant polytope $\Delta$. Suppose that $v>0$. Then $(Z,[\omega],\T)$ is $(v,w)$-uniformly K-stable if and only if $\Delta$ is $(v,w)$-uniformly K-stable.
\end{lemma}

\subsection{The space of dominating toric test configurations}{\label{a1:toric}}

Consider a polytope $\Delta_{R-f}$ associated with a function $f\in\mathcal{DPL}_{\mathrm{dom}}(\Delta)$ as in Section \ref{COnstruction:Donaldson}. Since the normal fan of $\Delta_{R-f}$ refines that of $\Delta\times[0,\sup(R-f)]$, the horizontal upper ray of the product fan remains a ray of $\Sigma_{R-f}$. Hence $f$ has a constant affine piece $A$, and $\Delta_{R-f}$ has a facet contained in $\{y=R-A\}$.

The correspondence between toric morphisms and morphisms of fans (see \cite[Theorem 3.3.4]{CLS11}) therefore gives a proper birational toric morphism
\[
\Pi:\mathcal{Z}_f\longrightarrow Z\times\mathbb{P}^1.
\]
This morphism is an isomorphism over $\mathbb{P}^1\setminus\{0\}$, since the fans agree over the punctured base. In particular, $(\Z_f,\A_{\Z_f})$ is a smooth dominating toric test configuration.

\subsection{Smooth dominating toric test configurations are enough}

Let $\mathcal{S}$ be one of the spaces $\RPL$, $\DPL$, or $\mathcal{DPL}_{\mathrm{dom}}(\Delta)$.

\begin{lemma}{\label{l:smooth:test}}
Suppose that $v\geq 0$ on $\Delta$, with $v>0$ on its interior, and that both $v$ and $vw$ extend continuously to $\Delta$. Then, the following statements are equivalent.
\begin{enumerate}
\item The Delzant polytope is $(v,w)$-weighted uniformly K-stable, i.e., satisfies Definition \ref{d:stability}.
\item There exists $\lambda>0$ such that for all $f\in\mathcal{S}$,
\begin{equation*}
\DF(f)\geq\lambda\lVert f\rVert_{J_v(\Delta)}.
\end{equation*}
\end{enumerate}
\end{lemma}

\begin{proof}
When $\mathcal{S}=\RPL$ and $v>0$, the equivalence follows from \cite[Theorem 1.0.1]{NS}. Here, we give a different argument based on uniform approximation, which also applies when $v$ vanishes on the boundary of $\Delta$.

We first explain why condition (2) for $\mathcal{S}=\mathcal{DPL}_{\mathrm{dom}}(\Delta)$ implies condition (2) for $\mathcal{S}=\RPL$; the converse direction is straightforward. This also proves the result for $\mathcal{S}=\DPL$.

Let $f=\max(f_1,\ldots,f_p)\in\RPL$. We choose an affine piece $f_{i_0}$ which agrees with $f$ on an open subset of $\Delta$ and set $f'=f-f_{i_0}$. Then $f'\geq0$ and $f'$ has the constant affine piece $0$. Since the weighted Donaldson--Futaki functional vanishes on affine functions and the weighted $J$-norm is invariant under their addition, it suffices to approximate $f'$.

Let $\Sigma_{R-f'}$ be the normal fan of $\Delta_{R-f'}$ and let $\Sigma_{\mathrm{prod}}$ be that of $\Delta\times[0,R]$, with $R>\sup_\Delta f'$. We consider a common projective refinement of these two fans. The corresponding toric model dominates both $\mathcal{Z}_{f'}$ and $Z\times\P^1$. Since the original fans agree over the punctured base, the refinement can be chosen to leave this open subset unchanged. By the toric resolution theorem \cite[Theorem 11.1.9]{CLS11}, we may further perform projective star subdivisions to obtain a smooth projective fan $\widetilde\Sigma$, still unchanged over the punctured base.

Let $\pi:\widetilde{\mathcal Z}\to\mathcal{Z}_{f'}$ be the resulting projective toric morphism. The pullback $\pi^*\mathcal A_{\mathcal Z_{f'}}$ is  not necessarily K\"ahler. We choose a torus-invariant divisor $D$ which is relatively ample over $\mathcal{Z}_{f'}$ and supported over the central fiber. For all sufficiently small rational $\varepsilon>0$, the class
\[
\mathcal A_\varepsilon=\pi^*\mathcal A_{\mathcal Z_{f'}}+\varepsilon[D]
\]
is K\"ahler and restricts to the original class on the general fiber.

In terms of the associated polytopes, the projective subdivisions can be realized by small truncations of faces of $\Delta_{R-f'}$ (see \cite[Sections 3.3, 6.2 and 11.1]{CLS11}). Since the modifications are supported over the central fiber, the original lower and side facets remain unchanged. The resulting polytope is therefore of the form
\[
\Delta_{R-f_\varepsilon}
=\{(x,y)\in\Delta\times\mathbb R\mid0\leq y\leq R-f_\varepsilon(x)\}
\]
for a rational convex piecewise-linear function $f_\varepsilon$, with
\[
\lim_{\varepsilon\to0}\|f_\varepsilon-f'\|_{C^0(\Delta)}=0.
\]
Moreover, $\Delta_{R-f_\varepsilon}$ is Delzant, its normal fan refines $\Sigma_{\mathrm{prod}}$, and its horizontal upper facet persists for sufficiently small $\varepsilon$. Hence $f_\varepsilon\in\mathcal{DPL}_{\mathrm{dom}}(\Delta)$.

We now conclude by uniform approximation. Since $v$ and $vw$ are continuous on the compact polytope $\Delta$, both $\DF$ and $\lVert\cdot\rVert_{J_v(\Delta)}$ are continuous with respect to uniform convergence. Indeed, for any continuous functions $f,g$ on $\Delta$,
\begin{equation}{\label{con:f}}
|\DF(f)-\DF(g)|\leq
\left(2\int_{\partial\Delta}v\,d\sigma_\Delta+\int_\Delta|vw|\,dx\right)
\|f-g\|_{C^0(\Delta)},
\end{equation}
and
\begin{equation}{\label{con:nj}}
\left|\|f\|_{J_v(\Delta)}-\|g\|_{J_v(\Delta)}\right|
\leq2\left(\int_\Delta v\,dx\right)\|f-g\|_{C^0(\Delta)}.
\end{equation}

Thus condition (2) holds for $f_\varepsilon$, with a constant independent of $\varepsilon$. Passing to the limit and using the invariance under addition of affine functions, we obtain
\[
\DF(f)\geq\lambda\lVert f\rVert_{J_v(\Delta)}
\]
for every $f\in\RPL$. This proves that condition (2) for $\mathcal{S}=\mathcal{DPL}_{\mathrm{dom}}(\Delta)$ implies condition (2) for $\mathcal{S}=\RPL$.

\smallskip

Finally, every function in $\mathcal{CV}^{\infty}(\Delta)$ can be approximated uniformly on $\Delta$ by functions in $\RPL$; see \cite[Propositions 2.5.23 and 2.5.24]{BPS14}. Conversely, every function in $\RPL$ can be uniformly approximated by smooth convex functions via standard mollification. By continuity of $\DF$ and $\lVert\cdot\rVert_{J_v(\Delta)}$ with respect to the $C^0(\Delta)$ topology, as established in \eqref{con:f} and \eqref{con:nj}, condition (1) is therefore equivalent to condition (2) for $\mathcal{S}=\RPL$. This concludes the proof.

\end{proof}

\section{Toric manifolds with a rigid and semisimple action}{\label{a:toric:comp}}

In this appendix, we explain how to extend Theorem \ref{t:exi:x} and Proposition \ref{l:kstab:stab:delta} to toric manifolds with a rigid and semisimple action in the terminology of \cite{ACGT06, ACGT04} (see Definitions \ref{d:rigide}). We indicate only the necessary modifications.

Let $(X^n,\omega,T)$ be a toric Kähler manifold with labeled Delzant polytope $(\hat{\Delta},\hat{\mathbf{L}})$. Suppose that $(X,\omega,T)$ is endowed with a rigid action of a subtorus $\mathbb{T} \subset T$, see Definition \ref{d:rigide}. It is known that the image of the moment map associated with $\T$ is a labeled Delzant polytope $(\Delta,\mathbf{L})$ (\cite[Prop. 4]{ACGT04}). By \cite{ACGT11}, $X^0 := \mu^{-1}(\mathring{\Delta})$ is a principal $\mathbb{T}^{\mathbb{C}}$-bundle over a base $B$:

$$
    \mathring{X} = P \times_\T \mathbb{T}^{\mathbb{C}} \longrightarrow B.
$$

Let us consider a particular class of semisimple $\T$-actions in the sense of Definition \ref{d:semisimple}. Assume that the full principal bundle arises from the \textit{blow-down} construction of \cite{ACGT11}. In other words, the base $B$ is a product $(B,\omega_B) := \prod_{b \in \beta}(B_b,\omega_b)$, where each factor $(B_b,\omega_b)$ is a complex projective space $\mathbb{P}^{d_b}$ endowed with the Fubini–Study metric $\omega_{b}$. Moreover, suppose that $P$ is endowed with a \textit{semisimple connection} $\theta \in \Omega^1(P,\mathfrak{t})$ of the form

$$
d\theta = \sum_{b \in \beta} \omega_b \otimes p_b,
$$

where $p_b \in \mathfrak{t}$ are the inward normal vectors of the facets $F_b$ of $\Delta$ involved in the blow-down process (see e.g. \cite[Section 3.1]{ACGT11}). For each facet $L_b \in \mathbf{L}$ involved in the blow-down process, we write $L_b = \langle p_b, x \rangle + c_b$.

We have an analogue of Lemma \ref{l:label}:

\begin{lemma}{\label{l:a:label}}
Let $\mathbf{L}=\{L_1, \dots, L_d\}$ be the label of $\Delta$ and  $ \mathbf{L}_b=\{L^b_{k_1}, \dots, L^b_{k_b}\}$ the label of $\Delta_b$. Then the label $\hat{\mathbf{L}}=(\hat{L}_1,\dots, \hat{L}_p)$ of the Delzant polytope $\hat{\Delta}$  is

\begin{equation}{\label{lab:a:delta:hat}}
 \hat{\Delta}= \{ L_i(x) \geq 0  \textnormal{ for } i \notin \beta, ( \langle p_b , x \rangle + c_b) L^b_{r_b} \geq 0 \textnormal{ for } b \in  \beta \},
\end{equation}
 
\end{lemma}

Since we do not have an explicit expression for the label $\mathbf{L}$, we need to adopt a different strategy than the one used in Lemma \ref{l:label}.

As observed in \cite{ACGL}, by \cite[Theorem 2]{Abr} and \cite[Remark 4.2]{ACGT04}, a strictly convex function $u$ on a labeled polytope $(\Delta, \textbf{L})$ which is smooth in the interior $\mathring{\Delta}$ defines a symplectic potential if and only if

\begin{equation}\label{boundary:cond}
\left\{
\begin{aligned}
    &u - \frac{1}{2}\sum_{i=1}^k L_i \log(L_i) \in \mathcal{C}^{\infty}(\Delta), \\
    &\det \textnormal{Hess}(u) \prod_{i=1}^k L_i \in \mathcal{C}^{\infty}(\Delta) \textnormal{ and is strictly positive}.
\end{aligned}
\right.
\end{equation}
The following proof is a generalization of \cite[Proposition 7]{ACGL}. 

\begin{proof}
Lemma \ref{Lemma-symp-pot} still holds in this context and we consider $\hat{u}$, a compatible symplectic potential in $\mathcal{S}(\hat{\Delta}, \hat{\mathbf{L}})$. It suffices to show that $\hat{u}$ satisfies the boundary conditions \eqref{boundary:cond} for the labeling on the RHS of \eqref{lab:delta:hat}. Indeed, by the uniqueness of the Kähler structure determined by a symplectic potential, the labeling given by the RHS of \eqref{lab:delta:hat} must correspond to $(X, \omega, T)$.

We recall that by definition $L_b=\langle p_b, x \rangle+c_b$ for any $b \in \beta$. 
We check the first condition of \eqref{boundary:cond}. By Lemma \ref{Lemma-symp-pot} we get that
\begin{align*}
   \hat{u} - \frac{1}{2} \sum_{c=1}^p \hat{L}_c \log\left( \hat{L_c} \right)  =& u + \sum_{b \in \beta} \left( \langle p_b, x \rangle+c_b \right) u_b -  \frac{1}{2} \sum_{i \notin \beta } L_i \log \left( L_i \right) \\
   &- \frac{1}{2} \sum_{b \in \beta} \sum_{j=1}^{k_b} L^b_{k_j} (\langle p_b, x \rangle +c_b) \log \left( L^b_{k_j} (\langle p_b, x \rangle +c_b) \right) \\
=& u - \frac{1}{2}  \sum_{i=1}^d L_i \log \left( L_i \right) +   \frac{1}{2} \sum_{b \in \beta}  (\langle p_b, x \rangle+c_b) \log \left(\langle p_b, x \rangle+c_b \right) \\
& +  \sum_{b \in \beta} \left( \langle p_b, x \rangle+c_b \right) u_b \\
&-  \frac{1}{2}  \sum_{b \in \beta} \sum_{j=1}^{k_b} L^b_{k_j} (\langle p_b, x \rangle +c_b) \log \left( L^b_{k_j} (\langle p_b, x \rangle +c_b) \right) \\
=& \textnormal{ smooth } +  \frac{1}{2} \sum_{b \in \beta}  (\langle p_b, x \rangle+c_b) \log \left(\langle p_b, x \rangle+c_b \right)  \\
&+\sum_{b \in \beta} \left( \langle p_b, x \rangle+c_b \right) \left( u_b -  \frac{1}{2}  \sum_{j=1}^{k_b}L^b_{k_j} \log \left( L_{k_j}^b  \right) \right)\\
&-  \frac{1}{2}  \sum_{b \in \beta } \sum_{j=1}^{k_b} L^b_{k_j} ( \langle p_b,x \rangle + c_b) \log ( \langle p_b,x \rangle + c_b) \\
=& \textnormal{ smooth } \\
&- \frac{1}{2} \sum_{b \in \beta} \left(  ( \langle p_b,x \rangle + c_b) \log ( \langle p_b,x \rangle + c_b)  \left( \sum_{j=1}^{k_b} L^b_{k_j}  - 1 \right) \right) \\
=&  \textnormal{ smooth }. 
\end{align*}  
\noindent To pass from the second line to the third, and from the third line to the fourth, we use the first condition of \eqref{boundary:cond}. To pass from the fourth line to the last line, we use the fact that for any $b \in \beta$, we have $\sum_{j=1}^{k_b} L^b_{k_j} =1$, since $(L_{k_j}^b)$ is the label of the projective space $(B_b,\omega_{b},T_b)$.

We now check the second condition of \eqref{boundary:cond}.
\begin{equation*}
  \begin{split}
    \det \hess (\hat{u}) \prod_{c=1}^p \hat{L}_c =&  \det \hess (u)  \prod_{b\in \beta} (\det  \hess (u_b )) (\langle p_b,x \rangle + c_b)   \prod_{c=1}^p \hat{L}_c\\
    =& \det \hess(u)\prod_{b\in \beta} \det  \hess (u_b) (\langle p_b,x \rangle + c_b) \prod_{j=1}^{k_b} L_{k_j}^b \prod_{i \notin \beta} L_i \\
    =& \left( \det \hess (u) \prod_{i}^d L_i \right)  \left(  \prod_{b \in \beta} \det \hess(u_b) \prod_j^{k_b} L^b_{k_j} \right).
\end{split}   
\end{equation*}

\noindent For passing from the second line to the third line we use that $L_b=\langle p_b,x \rangle + c_b$ for any $b \in \beta$. By the second condition of \eqref{boundary:cond}, the two terms of the third line are positive and smooth on $\hat{\Delta}$. That concludes the proof.
\end{proof}

\begin{lemma}{\label{l:a:bound:measure}}
The measure $d\sigma_{\hat{\Delta}}$ is given by

\begin{equation*}
\begin{split}
&d \sigma_{\hat{\Delta}}|_{\{L_i=0\}} = (pd\sigma_{\Delta})|_{\{L_i=0\}}  dx_B  \\
&d \sigma_{\hat{\Delta}}|_{\{\langle p_b , x \rangle + c_b) L^b_{r_b}=0\}} = \frac{p}{\langle p_b , x \rangle + c_b} dx (d\sigma_{\Delta_B})_{\{L^b_{r_b}=0\}}.
\end{split}    
\end{equation*} 

\end{lemma}

\begin{proof}
The proof goes as in Lemma \ref{l:bound:measure}.
\end{proof}

As in Section \ref{s:compatible:test}, we can define $\T$-compatible  test configurations. For any function $f \in \mathcal{DPL}_{\mathrm{dom}}(\Delta)$, we consider the associated polytope $\Delta_{R - f}$ as explained in Section \ref{COnstruction:Donaldson}. 

\begin{prop}{\label{l:a:CT}}
The function $f$ belongs to $\mathcal{DPL}_{\mathrm{dom}}(\hat{\Delta})$. In particular,  $f$ induces a smooth dominating test configuration $(\X_{f}, \A_{\X_f})$ for $(X,[\omega],\T)$ defined as the toric test configuration $(\X_{f}, \A_{\X_{{f}}})$.
\end{prop}

\begin{proof} 
The proof is similar to that of Proposition \ref{l:CT}, treating separately the affine functions that arise from the blow-down and those that do not. 
\end{proof}

As a final step, applying the same arguments as in Theorem \ref{t:exi:x} and Proposition \ref{l:kstab:stab:delta}, but considering the Kähler manifold $(Z,[\omega_Z],\T)$ associated with $\Delta$ by the Delzant correspondence (instead of the projective space $\mathbb{P}^\ell$), we obtain the following.

\begin{theorem}
Let $v>0$ and $w$ be two weight functions on $\Delta$. Suppose that $(X,[\omega],\T)$ is $(v,w)$-uniformly K-stable for compatible  test configurations. Then there exists a weighted $(v,w)$-cscK metric in $[\omega]$. Moreover, this metric is compatible. 
\end{theorem}

\section{Preliminaries on horospherical varieties}{\label{a:horo}}

\subsection{Basics of $G^{\mathbb{C}}$-varieties}
We recall some basic definitions about  $G^{\mathbb{C}}$-varieties in this section.

Assume that $G$ is a connected compact Lie group and let $G^{\mathbb{C}}$ be its complexification, a connected reductive algebraic group. Assume that $B$ is a Borel subgroup of $G^{\mathbb{C}}$. Assume that $X$ is a normal $G^{\mathbb{C}}$-variety. Let's denote the field of rational functions on $X$ as $\mathbb{C}(X)$. The group $G^{\mathbb{C}}$ acts on $\mathbb{C}(X)$ as $(g\cdot f)(x) = f(g^{-1}\cdot x)$, where $g\in \Gc$ and $f\in \mathbb{C}(X)$. Let's denote the character group of $B$ by $\mathcal{X}(B)$. The following definition is classical (see \cite[Section 1.1]{brion1997varietes}),
\begin{define}
We have the weight lattice of $X$ with respect to $B$
\begin{align*}
M = \{\chi \in \mathcal{X}(B) \mid \exists f_{\chi} \in \mathbb{C}(X)^{*} \text{ such that } 
 b \cdot f_{\chi} = \chi(b) f_{\chi} \text{ for any } b \in B \}.
\end{align*}
\end{define}

Now let's assume that we have an ample $G^{\mathbb{C}}$-linearized line bundle on a projective $G^{\mathbb{C}}$-variety $X$, which means that there is a lift of the $G^{\mathbb{C}}$-action of $X$ to $L$. The following definition is classic.
\begin{define}
For $k\in \mathbb{Z}_{>0}$, let
\[\Delta_{\mathrm{alg},k}  = \left\{\frac{\chi}{k}|\exists\text{ nonzero } s_{\chi}\in H^{0}(X,kL)\text{ such that }  b\cdot s_{\chi} = \chi(b)s_{\chi} \text{ for any }b\in B\right\}\]
Then let
    \[\Delta_{\mathrm{Qalg}} = \cup_{k=1}^{\infty} \Delta_{\mathrm{alg},k}.\]
It is a rational polytope inside $\mathcal{X}(B)\otimes \mathbb{Q}$. If we take any element $\chi$ inside $\Delta_{\mathrm{Qalg}}$, then the polytope $\Delta_{\mathrm{Qalg}} - \chi$ is a full dimensional polytope inside $M\otimes \mathbb{Q}$. See \cite[Section 1.1]{brion1997varietes}.
We let $\Delta_{\mathrm{alg}} = \overline{\Delta_{\mathrm{Qalg}}}$ be the closure of $\Delta_{\mathrm{Qalg}}$ in $\mathcal{X}(B)\otimes \mathbb{R}.$ We call $\Delta_{\mathrm{alg}}$ the algebraic moment polytope of $L.$
\end{define}

On the other hand, we have the Kirwan polytope from the symplectic side. 

\begin{define}
Let \((X, \omega)\) be a compact symplectic manifold equipped with a Hamiltonian action of a compact connected Lie group $G$, and let \(\mu_{\omega}^{G}: X \to \mathfrak{g}^*\) be the associated moment map, where \(\mathfrak{g}^*\) is the dual of the Lie algebra of $G$.

The Kirwan polytope is the convex polytope defined as
\[
\Delta_{\mathrm{Kir}} = \mu_{\omega}^{G}(X) \cap \hat{\mathfrak{t}}^*_+,
\]
where \(\hat{\mathfrak{t}}^*\) is the dual of the Lie algebra of a maximal torus \(T\subset G\), and \(\hat{\mathfrak{t}}^*_+ \subset \hat{\mathfrak{t}}^*\) is a chosen positive Weyl chamber. Notice that $\hat{\mathfrak{t}}^*$ can be identified with the $T$-invariant subspace of $\mathfrak{g}^*.$
\end{define}

The relation between the algebraic moment polytope and the Kirwan polytope is explained in the following paragraphs.

Let $\mathcal{M}$ be a finite-dimensional $G^{\mathbb{C}}$-module. We choose a maximal complex torus $T^{\mathbb{C}}$ in $G^{\mathbb{C}}$ such that its intersection with $G$, $T^{\mathbb{C}} \cap G = T$, is a maximal compact real torus in $G$. We also fix a Borel subgroup of $G^{\mathbb{C}}$ that contains $T^{\mathbb{C}}$. 

We have a natural $G^{\mathbb{C}}$-action on the projective space $\mathbb{P}(\mathcal{M})$. There is also a canonical $G^{\mathbb{C}}$-action on the line bundle $\mathcal{O}_{\mathbb{P}(\mathcal{M})}(-1)$, induced from the action on $\mathcal{M}$. 

From this, we get a $G^{\mathbb{C}}$-action on the dual line bundle $\mathcal{O}_{\mathbb{P}(\mathcal{M})}(1)$, defined by  
\[
(g \alpha)(g v) = \alpha(v)
\]  
where $\alpha \in \mathcal{O}_{\mathbb{P}(\mathcal{M})}(1)_m$, $v \in \mathcal{O}_{\mathbb{P}(\mathcal{M})}(-1)_m$, and $m$ is any point in $\mathbb{P}(\mathcal{M})$. 

If $X \subset \mathbb{P}(\mathcal{M})$ is an irreducible closed subvariety that is stable under the $G^{\mathbb{C}}$-action, then $\mathcal{O}_X(1)$ inherits a natural $G^{\mathbb{C}}$-linearization. Thereafter, we have the polytopes $\Delta_{\mathrm{Qalg}}$ and $\Delta_{\mathrm{alg}}$.

Next, we choose a $G$-invariant Hermitian metric $\langle \cdot, \cdot \rangle$ on $\mathcal{M}$, then we have a corresponding Fubini-Study metric. We also have a moment map with respect to the restriction of the Fubini-Study metric $\omega_{\mathrm{FS}}$ on $X$, $\mu_{\omega_{\mathrm{FS}}}^{G}:X\rightarrow \mathfrak{g}^{*}$. The moment map is given by $\mu_{\omega_{\mathrm{FS}}}^{G}(x)(a) = \frac{\langle a\tilde{x},\tilde{x}\rangle}{i\langle \tilde{x},\tilde{x}\rangle}$, where $a\in \mathfrak{g}$, $\tilde{x}$ is any vector in $\mathcal{M}$ over $x$. We denote the Lie algebra of $T$ as $\hat{\mathfrak{t}}$. The dual $\hat{\mathfrak{t}}^{*}$ is identified with the $T$-invariant subspace of $\mathfrak{g}^{*}$. The Borel subgroup $B$ of $G^{\mathbb{C}}$ provides a choice of positive Weyl chamber inside $\hat{\mathfrak{t}}^{*}$, which we denote as $\hat{\mathfrak{t}}_{+}^{*}$. Then we have the Kirwan polytope $\Delta_{\mathrm{Kir}}= \mu_{\omega}^{G}(X) \cap \hat{\mathfrak{t}}^*_+.$ We rely heavily on the following important theorem in Section \ref{horo tc for proj}.

\begin{theorem}\textbf{\cite{brion1987image}}\label{kir_alg_id}
The polytope $-w_{0}\Delta_{\mathrm{Qalg}}$ is exactly the set of rational points of $\Delta_{\mathrm{Kir}}$, where $w_{0}$ is the longest element of the Weyl group of $(G,T)$. 
\end{theorem}

\subsection{Horospherical varieties}\label{preliminarieshoro}
We have the following crucial class of $G^{\mathbb{C}}$-varieties.
\begin{define}
A normal $G^{\mathbb{C}}$-variety is spherical if it contains an open $B$-orbit, where $B$ is a Borel subgroup.
\end{define}

We recall the following definition about a special type of spherical varieties from \cite{pasquier2006fano,pasquier2008varietes}.

\begin{define}
    A closed subgroup $H$ of $G^{\mathbb{C}}$ is called horospherical if it contains the unipotent radical of a Borel subgroup $B$. We call $G^{\mathbb{C}}/H$ a horospherical homogeneous space. A normal $G^{\mathbb{C}}$-equivariant embedding $X$ of $G^{\mathbb{C}}/H$ is called a horospherical variety.
\end{define}
Let $N_{G^{\mathbb{C}}}(H)$ be the normalizer of $H$. We collect here a few facts,

\begin{enumerate}
    \item $N_{G^{\mathbb{C}}}(H)$ is a parabolic subgroup which contains $B$;
    \item $N_{G^{\mathbb{C}}}(H)/H$ is a complex torus;
    \item the spherical lattice $M$ is isomorphic to the character group of $N_{G^{\mathbb{C}}}(H)/H$ (see \cite[Section 4.3]{brion1997varietes});
    \item note that $N_{G^{\mathbb{C}}}(H)/H$ acts on $G^{\mathbb{C}}/H$ by $pH\cdot gH = gp^{-1}H$, where $p\in N_{G^{\mathbb{C}}}(H)$. This group action realizes $G^{\mathbb{C}}/H$ as a torus principal bundle over $G^{\mathbb{C}}/N_{G^{\mathbb{C}}}(H)$. This action extends to $X$ and realizes $N_{G^{\mathbb{C}}}(H)/H$ as the identity component of $\mathrm{Aut}_{G^{\mathbb{C}}}(X)$ (\cite[Section 3.1.3]{delcroix2020k}).
\end{enumerate}

The following is a significantly simplified version of \cite[Theorem~4.1]{Del2023}. We only require this simplified form, as we deal exclusively with horospherical varieties in this work.

\begin{theorem}\label{sphericaltest}  \textbf{\cite[Theorem 4.1]{Del2023}}
For a polarized horospherical variety \((X,L)\), \(G^{\mathbb{C}}\)-equivariant test configurations are in one-to-one correspondence with strictly positive rational piecewise linear concave functions on \(\Delta_{\mathrm{alg}}\).
\end{theorem}

\section{Complex structures and closed forms on bundles}\label{bundleconstruction}
\subsection{Complex structure on bundles}\label{cmlxstrcons}
This subsection covers classical results, which we include for completeness and to fix notations. We omit the proofs here, as they are explained in \cite[Section 1.3]{yin2025explicit}.

Let \( C \) be a smooth curve with complex structure $J$ and \( \pi:P \to C \) be a $G$-principal bundle, where $G$ is a compact connected Lie group. Let \( \theta \) be a connection on \( P \), and \( \Omega \) be its curvature. Since $C$ is a complex curve, for any local section \( s: D \to P \), the pulled-back curvature form \( s^* \Omega \) is automatically of type \((1,1)\).

The connection provides a horizontal distribution \( \mathcal{H} \subset TP \). We have a section $J_{\mathcal{H}}$ of \( \text{End}(\mathcal{H}) \) which is the lift of \( J \), it satisfies \( J_{\mathcal{H}}^2 = -\text{Id} \). Then we can define \( \mathcal{H}^{1,0} \subset \mathcal{H}^{\mathbb{C}} \subset T^{\mathbb{C}} P \), where $\mathcal{H}^{1,0}$ is the $i$-eigenvalue space of $J_{\mathcal{H}}$. 

The following proposition is classic.

\begin{prop}
    The $(1,0)$-part of the horizontal distribution, namely $\mathcal{H}^{1,0}$, defines a CR structure  on $P$.
\end{prop}
Now assume that we have a complex manifold \( X \) on which $G$ acts in a holomorphic way, which means that we have a Lie group homomorphism $G\rightarrow \mathrm{Aut}(X)$, where $\mathrm{Aut}(X)$ is the automorphism group of $X$ as a complex manifold. Consider \( P \times X \) and \( P \times_{G} X \), where \( P \times_{G} X \) is the quotient of \( P \times X \) by the following right group action,
\[ (p,x)\cdot g = (pg,g^{-1}x).\]

Naturally we have \( P \times X \to P \times_{G} X \). Actually, $P \times X \to P \times_{G} X$ is a principal $G$-bundle  with the right $G$-action given above.

We define the horizontal distribution of the above $G$-principal bundle as \( \mathcal{H} \oplus TX \).
This provides a connection. It is clear that \( \mathcal{H}^{1,0} \oplus T^{1,0}X \) is a CR-structure and is $G$-invariant. We also have a section $J_{\mathcal{H}\oplus TX}$ of \( \text{End}(\mathcal{H} \oplus TX) \) which is clearly $G$-invariant. So $J_{\mathcal{H}\oplus TX}$ descends to an almost complex structure \( J_{Y} \) on \( P \times_{G} X =: Y \). 

The following lemma is classic.

\begin{lemma}
The almost complex structure $J_{Y}$ is integrable.    
\end{lemma}

Note that the image of $\mathcal{H}$ defines a subbundle of $TY$, which we call the $G$-horizontal distribution of $TY$, or, with a slight abuse of terminology, the $G$-horizontal distribution of $Y$. The image of $TX$ is also a subbundle of $TY$, it is simply the vertical distribution of the fibration $Y\rightarrow C.$

\subsection{Form constructions}\label{Form_constructions}

This subsection reproduces the material in \cite[Proposition 3.3]{mccarthy2022canonical}, which in turn, essentially restates a theorem of Weinstein (see \cite[Theorem 6.3.3]{mcduff2017introduction}). We include the subsection here for completeness and to fix notations. More explanations can be found in \cite[Section 1.3]{yin2025explicit}.

Assume that we have a $G$-pre-Hamiltonian manifold \( (X, \omega_{X}) \) with a moment map
\[\mu_{\omega}^{G}: X \rightarrow \mathfrak{g}^*.\]

Consider the product \( P \times X \) and the quotient space \( P \times_{G} X =: Y \), we have
\[
P \times X \rightarrow Y \rightarrow C.
\]

As before, \( P \times X \rightarrow Y \) is a principal $G$-bundle with the $G$-action defined by:
\[
(p,x) \cdot g = (pg, g^{-1}x).
\]

We decompose 
\[
T(P \times X) = \mathcal{H} \oplus TX \oplus \mathcal{V},
\]
where \( \mathcal{V} \) represents the vertical distribution of the principal $G$-bundle $P\times X\rightarrow Y$. As before, \( \mathcal{H} \oplus TX \) provides a connection.

On \( P \times X \), we have a form \( \langle \mu_{\omega}^{G}, \Omega \rangle \). This form is $G$-invariant because both 
 $\mu_{\omega}^{G}$ and $\Omega$ are $G$-equivariant. It's clear that \( \langle \mu_{\omega}^{G}, \Omega \rangle \) vanishes on the vertical distribution of the bundle \( P \times X \rightarrow Y \). This tells us that \( \langle \mu_{\omega}^{G}, \Omega \rangle \) is basic, so it is a pullback from \( Y \). The corresponding form on $Y$ vanishes on the vertical distribution of the fibration of $Y\rightarrow C.$

We also consider the form \( \omega_{X} \) (with a slight abuse of notation) on \( P \times X \) which is \( \omega_{X} \) on \( TX \) but vanishes on \( \mathcal{H} \) and \( \mathcal{V} \). This \( \omega_{X} \) is $G$-invariant, it is also a basic form and hence a pullback from \( Y \). The corresponding form on $Y$ vanishes on the $G$-horizontal distribution.

The following lemma is classic and is included in the proof of \cite[Proposition 3.3]{mccarthy2022canonical}.

\begin{lemma}
    The form $\langle \mu_{\omega}^{G}, \Omega \rangle+\omega_{X}$ is closed on $P\times X$, so it descends to a closed form on $Y$.
\end{lemma}

With a bit of abuse of notations, we still denote the form on $Y$ as $\langle \mu_{\omega}^{G}, \Omega \rangle+\omega_{X}$ in this article.

\subsection{Projective bundle as an example}\label{projective_bundle_G_construction}

We assume in this subsection that $\omega_{C}$ has integral $2\pi$ and is a constant scalar curvature Kähler metric.

Let $(E_{j},h_{j})$ be an Hermitian-Einstein vector bundle over a smooth curve $C$ for $j = 0,\cdots, \ell,$ let $E = \oplus_{j=0}^{\ell}E_{j}.$ 
Let $G = \prod_{j=0}^{\ell}U(d_j)$. Let $P$ be the corresponding projectively-flat principal $G$-bundle over $C$ with the curvature in the form:
\[
\Omega = - i 
\begin{pmatrix}
\mu(E_{0}) \text{Id} & \cdots & 0 \\
\vdots & \ddots & \vdots \\
0 & \cdots & \mu(E_{l}) \text{Id}
\end{pmatrix} 
\otimes \omega_C,
\]
where $\mu(E_{j})$ is the slope of $E_{j}.$

Let $X $ be $  \mathbb{P}\left(\oplus_{j=0}^\ell \mathbb{C}^{d_j}\right)$, the projective space of lines. For an element in $X$, we write it as 
\[
[\mathbf{z}_0 : \mathbf{z}_1 : \dots : \mathbf{z}_\ell].
\]
where $\mathbf{z}_i \in \mathbb{C}^{d_i}$.
We endow $X$ with the natural Fubini-Study metric $\omega_{\mathrm{FS}}$. We have a natural $\prod_{j=0}^{\ell}U(d_j)$-action on $X$. The corresponding moment map is:
\[
\mu^{G}_{\omega_{\mathrm{FS}}}(x)(a) = \frac{\langle a\tilde{x},\tilde{x}\rangle}{i\langle \tilde{x},\tilde{x}\rangle},
\]
where $a \in \oplus_{j=0}^{\ell}\mathfrak{u}(d_{j})$ and $\tilde{x}$ is a point in $\mathbb{C}^{d_0} \oplus \cdots \oplus \mathbb{C}^{d_\ell}$ over $x$. In this case, we see that:
\[\langle \mu^{G}_{\omega_{\mathrm{FS}}}, \Omega \rangle = 
-\frac{\mu(E_0) |\mathbf{z}_0|^2 + \mu(E_1) |\mathbf{z}_1|^2 + \cdots + \mu(E_\ell) |\mathbf{z}_\ell|^2}{|\mathbf{z}_0|^{2}+\cdots |\mathbf{z}_\ell|^2}\omega_{C} .
\]
If $c >   \mu(E_{j})$ for all $j$, then $
c \omega_C + \langle \mu^{G}_{\omega_{\mathrm{FS}}}, \Omega \rangle + \omega_{FS} \text{ is a Kähler form on } Y = P \times_{G} X.$
Actually $Y$ is just $\mathbb{P}(E)$ in this case. Clearly the Kähler form is the same as the one constructed in \cite{AK}, and it lives in the Kähler class
\[
\left[ 2\pi(cF +  \mathcal{O}_{\mathbb{P}(E)}(1)) \right]
\]
on $\mathbb{P}(E)$, where $F$ represents a fiber. Note that $
\left[ 2\pi(cF +  \mathcal{O}_{\mathbb{P}(E)}(1)) \right]$ represents a Kähler class if and only if
$c-\mu(E_{j})>0$ for all $j.$
So we know that the Kähler classes of $c \omega_C + \langle \mu^{G}_{\omega_{\mathrm{FS}}}, \Omega \rangle + \omega_{FS}$ exhaust the Kähler cone of $\mathbb{P}(E)$ up to positive scaling.

\section*{Statements and Declarations}

\subsection*{Data availability}
Data sharing is not applicable to this article as no datasets
were generated or analysed during the current study.

\subsection*{Conflict of interest}
On behalf of all authors, the corresponding author states
that there is no conflict of interest.

\bibliographystyle{plain}
\bibliography{References.bib}
\end{document}